\documentclass[11pt,reqno]{amsart}

\usepackage[noruninsections]{preamble}
\mathtoolsset{showonlyrefs}
\makeatletter
\AtBeginDocument{%
  \let\label@in@display\cref@old@label@in@display
  \let\ltx@label\cref@old@label
}
\makeatother
\usepackage{latexsym}
\usepackage{fancybox}
\usepackage{eurosym}

\def\1{\raisebox{2pt}{\rm{$\chi$}}}
\renewcommand{\a}{\alpha}

\renewcommand{\R}{{\mathbb R}}

\renewcommand{\eps}{{\varepsilon}}

\def\vint_#1{\mathchoice%
  {\mathop{\kern 0.2em\vrule width 0.6em height 0.69678ex depth -0.58065ex
    \kern -0.8em \intop}\nolimits_{\kern -0.4em#1}}%
  {\mathop{\kern 0.1em\vrule width 0.5em height 0.69678ex depth -0.60387ex
    \kern -0.6em \intop}\nolimits_{#1}}%
  {\mathop{\kern 0.1em\vrule width 0.5em height 0.69678ex
    depth -0.60387ex \kern -0.6em \intop}\nolimits_{#1}}%
  {\mathop{\kern 0.1em\vrule width 0.5em height 0.69678ex depth -0.60387ex
    \kern -0.6em \intop}\nolimits_{#1}}}
\def\vintslides_#1{\mathchoice%
  {\mathop{\kern 0.1em\vrule width 0.5em height 0.697ex depth -0.581ex
    \kern -0.6em \intop}\nolimits_{\kern -0.4em#1}}%
  {\mathop{\kern 0.1em\vrule width 0.3em height 0.697ex depth -0.604ex
    \kern -0.4em \intop}\nolimits_{#1}}%
  {\mathop{\kern 0.1em\vrule width 0.3em height 0.697ex
    depth -0.604ex \kern -0.4em \intop}\nolimits_{#1}}%
  {\mathop{\kern 0.1em\vrule width 0.3em height 0.697ex depth -0.604ex
    \kern -0.4em \intop}\nolimits_{#1}}}

\newcommand{\aveint}[2]{\mathchoice%
  {\mathop{\kern 0.2em\vrule width 0.6em height 0.69678ex depth -0.58065ex
    \kern -0.8em \intop}\nolimits_{\kern -0.45em#1}^{#2}}%
  {\mathop{\kern 0.1em\vrule width 0.5em height 0.69678ex depth -0.60387ex
    \kern -0.6em \intop}\nolimits_{#1}^{#2}}%
  {\mathop{\kern 0.1em\vrule width 0.5em height 0.69678ex
    depth -0.60387ex \kern -0.6em \intop}\nolimits_{#1}^{#2}}%
  {\mathop{\kern 0.1em\vrule width 0.5em height 0.69678ex depth -0.60387ex
    \kern -0.6em \intop}\nolimits_{#1}^{#2}}}

\allowdisplaybreaks
\makeatletter
\renewcommand\paragraph{\@startsection{paragraph}{4}{\z@}%
  {.5\baselineskip \@plus 2\p@}{-.5em}{\normalfont\bfseries}}
\makeatother

{\catcode`p=12 \catcode`t=12 \gdef\eeaa#1pt{#1}}
\def\accentadjtext#1{\setbox0\hbox{$#1$}\kern
  \expandafter\eeaa\the\fontdimen1\textfont1 \ht0}
\def\accentadjscript#1{\setbox0\hbox{$#1$}\kern
  \expandafter\eeaa\the\fontdimen1\scriptfont1 \ht0}
\def\accentadjscriptscript#1{\setbox0\hbox{$#1$}\kern
  \expandafter\eeaa\the\fontdimen1\scriptscriptfont1 \ht0}
\def\accentadjtextback#1{\setbox0\hbox{$#1$}\kern
  -\expandafter\eeaa\the\fontdimen1\textfont1 \ht0}
\def\accentadjscriptback#1{\setbox0\hbox{$#1$}\kern
  -\expandafter\eeaa\the\fontdimen1\scriptfont1 \ht0}
\def\accentadjscriptscriptback#1{\setbox0\hbox{$#1$}\kern
  -\expandafter\eeaa\the\fontdimen1\scriptscriptfont1 \ht0}
\def\itoverline#1{{\mathsurround0pt\mathchoice
  {\rlap{$\accentadjtext{\displaystyle #1}
    \accentadjtext{\vrule height1.593pt}
    \overline{\phantom{\displaystyle #1}
    \accentadjtextback{\displaystyle #1}}$}{#1}}
  {\rlap{$\accentadjtext{\textstyle #1}
    \accentadjtext{\vrule height1.593pt}
    \overline{\phantom{\textstyle #1}
    \accentadjtextback{\textstyle #1}}$}{#1}}
  {\rlap{$\accentadjscript{\scriptstyle #1}
    \accentadjscript{\vrule height1.593pt}
    \overline{\phantom{\scriptstyle #1}
    \accentadjscriptback{\scriptstyle #1}}$}{#1}}
  {\rlap{$\accentadjscriptscript{\scriptscriptstyle #1}
    \accentadjscriptscript{\vrule height1.593pt}
    \overline{\phantom{\scriptscriptstyle #1}
    \accentadjscriptscriptback{\scriptscriptstyle #1}}$}{#1}}}}
\def\itunderline#1{{\mathsurround0pt\mathchoice
  {\rlap{$\underline{\phantom{\displaystyle #1}
    \accentadjtextback{\displaystyle #1}}$}{#1}}
  {\rlap{$\underline{\phantom{\textstyle #1}
    \accentadjtextback{\textstyle #1}}$}{#1}}
  {\rlap{$\underline{\phantom{\scriptstyle #1}
    \accentadjscriptback{\scriptstyle #1}}$}{#1}}
  {\rlap{$\underline{\phantom{\scriptscriptstyle #1}
    \accentadjscriptscriptback{\scriptscriptstyle #1}}$}{#1}}}}
\newcommand{\upperH}{\itoverline{H}}
\newcommand{\lowerH}{\itunderline{H}}
\newcommand{\Deltap}{\Delta_p}

\providecommand{\keywords}[1]{%
  \par\smallskip\noindent\textit{Key words and phrases.} #1}
\providecommand{\subjclass}[2][]{%
  \par\smallskip\noindent\textit{2020 Mathematics Subject Classification.} #2}

\title{Regularity of soda-can domains for the $p$-parabolic equation}
\author{Benny Avelin}
\address{Department of Mathematics, Uppsala University\\
  Box 480, 751 06 Uppsala, Sweden}
\author{Mikko Parviainen}
\address{Department of Mathematics and Statistics, University of Jyv\"askyl\"a\\
  P.O. Box 35, FI-40014 University of Jyv\"askyl\"a, Finland}
\date{October 5, 2026}

\begin{document}

\begin{abstract}
  We classify the regularity of the origin for the $p$-parabolic equation
  in the soda-can domains
  \begin{equation*}
    \Theta_{l,\theta}
    =\{(x,t)\in\R^n\times\R:0<-t<\theta |x|^l<\theta\},
    \qquad l,\theta>0.
  \end{equation*}
  The classification covers every $p>1$ in dimensions $n\ge2$.
  At the endpoint $l=p$, the origin is regular if and only if
  $p\ge2n/(n+1)$.  We also prove that the origin admits a traditional
  barrier for every $p>2$ and $l>0$.  When $2<p<n$ and $l<p$, the
  origin is nevertheless irregular.  This answers negatively the well known open
  question whether one barrier suffices to characterize boundary
  regularity in the degenerate range $p>2$.
\end{abstract}

\keywords{barrier family, boundary regularity, Perron method,
  $p$-parabolic equation, soda-can domain}
\subjclass[2020]{Primary 35K20; Secondary 31B25, 35B65, 35K65, 35K92}
\thanks{MP acknowledges funding from  the Research Council of Finland (project 360185). The proof development and the preparation of this note were assisted by ChatGPT.  The authors take full responsibility for the contents of the paper.
}
\maketitle

\tableofcontents

\clearpage

\section{Introduction}\label{sec:introduction}

We study boundary regularity for the $p$-parabolic equation
\begin{equation}\label{eq:p-parabolic}
  \partial_t u-\Deltap u=0,
  \qquad
  \Deltap u=\operatorname{div}(|\nabla u|^{p-2}\nabla u),
\end{equation}
with $1<p<\infty$.  The equation is degenerate when $p>2$, singular when $1<p<2$, and reduces to the heat equation when $p=2$.  A boundary point is regular if the corresponding upper Perron solution tends to the prescribed value for every continuous boundary datum.

Barriers are a standard tool for proving regularity.  A barrier is a positive supersolution that tends to zero at the given boundary point and has positive lower limits at every other boundary point.  For the heat equation, one barrier suffices.  Linearity makes every positive multiple a barrier, so its values away from the given point can be made arbitrarily large.  For $p\ne2$, multiplication by a constant does not in general preserve the supersolution inequality.  Indeed,
\begin{equation*}
  \partial_t(kw)-\Deltap(kw)=k\,\partial_t w-k^{p-1}\Deltap w,
  \qquad k>0.
\end{equation*}
The different powers of $k$ explain why a supersolution need not remain one after multiplication.  Bj\"orn--Bj\"orn--Gianazza--Parviainen \cite[Theorem~3.3]{BBGP15} proved that regularity is characterized by a family of barriers whose members become arbitrarily large on the boundary away from the given point.  The corresponding definition appears in \cref{def:barrier-family}.

Peter Lindqvist asked, cf.\ \cite{KL96}, whether one barrier can suffice for the $p$-parabolic equation.  The question was formulated in \cite[Section~8, Problem~1]{BBGP15}.  Bj\"orn--Bj\"orn--Gianazza \cite[Proposition~1.2]{BBG17} answered it negatively for $1<p<2$, but it has remained open in the degenerate range $p>2$. In this paper we settle this
question by constructing an irregular boundary point that admits a single traditional barrier when
$p>2$.  

The relevant domain is the soda-can domain
\begin{equation}\label{eq:soda-domain}
  \Theta_{l,\theta}
  =\{(x,t)\in\R^n\times\R\mid 0<-t<\theta |x|^l<\theta\},x
\end{equation}
where $n\ge2$ and $l,\theta>0$.  The origin lies on the top $t=0$, while the initial boundary $t=-\theta|x|^l$ approaches it from below, as shown in \cref{fig:soda-can}.  Larger $l$ gives a thinner domain near the origin.  For $p\ne2$, scaling shows that $\theta$ does not affect regularity \cite[Proposition~4.1]{BBG17}.  The exponent $l=p$ is invariant under the intrinsic scaling $(x,t)\mapsto(\lambda x,\lambda^pt)$ of the equation.

\begin{figure}[htb]
  \centering
  \begin{tikzpicture}[x=1.35cm,y=1.35cm,>=stealth,
      every node/.style={font=\small}]
    \begin{scope}
      \fill[black!10]
      (-1.55,0)
      -- plot[smooth,domain=-1.55:1.55,samples=80]
      (\x,{-1.35*(abs(\x)/1.55)^2})
      -- (1.55,0) -- cycle;
      \draw[thick]
      plot[smooth,domain=-1.55:1.55,samples=80]
      (\x,{-1.35*(abs(\x)/1.55)^2});
      \draw[thick]
      (-1.55,-1.35)--(-1.55,0)--(1.55,0)--(1.55,-1.35);
      \draw[dashed] (-1.55,-0.68)--(1.55,-0.68)
      node[right] {$t<0$};
      \draw[->] (-1.8,0)--(1.82,0) node[right] {$x_1$};
      \draw[->] (0,-1.55)--(0,0.36) node[above] {$t$};
      \fill (0,0) circle (1.25pt) node[above left] {$(0,0)$};
      \node at (-1.12,-0.30) {$\Theta_{l,\theta}$};
      \node[below left] at (1.35,{-1.35*(1.35/1.55)^2})
      {$t=-\theta |x|^l$};
    \end{scope}
  \end{tikzpicture}
  \caption{A schematic of the soda-can domain $\Theta_{l,\theta}$.  The name was suggested in \cite{BBGP15}, since the initial boundary resembles the bottom of a soda can.}
  \label{fig:soda-can}
\end{figure}
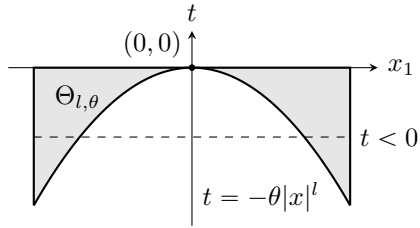

\subsection{Main results}

We classify regularity for all~$p>1$ and~$n\ge2$.  We write $p_*=2n/(n+1)$.

\begin{theorem}[Classification of regularity]\label{thm:classification}
  Let $n\ge2$, $1<p<\infty$, and $l,\theta>0$.  The origin is regular with
  respect to $\Theta_{l,\theta}$ if and only if one of the following holds:
  \begin{enumerate}
    \item $1<p<p_*$ and $l>p$
    \item $p_*\le p<n$ and $l\ge p$
    \item $p\ge n$, with no further restriction on $l$.
  \end{enumerate}
  When $p=2$, the statement is understood in the usual sense for the heat
  equation.
\end{theorem}

For $p<n$, the threshold $p_*$ determines regularity at the endpoint $l=p$.  The same value separates the supercritical and subcritical ranges in the local regularity theory for singular parabolic equations~\cite[Chapter~6]{DGV12}.

Our new cases complete the classification previously left open, to the best of the authors' knowledge.  For $2<p<n$, we prove irregularity when $l<p$ and regularity at $l=p$.  For $p=n>2$, we prove regularity for all $l>0$.  These results settle Open Problem~1.9(a) of Bj\"orn--Bj\"orn~\cite{BB26}.  For $p_*<p<2$, we prove irregularity for $(n-p)(2-p)/(p-1)\le l<p$, and for $p_*\le p<2$ we prove regularity at $l=p$, including the double-critical case $p=p_*$.  This answers their Open Problem~1.10.  The remaining cases follow from their Propositions~1.2--1.5 and~\cite[Proposition~4.2]{BBGP15}.

Our second result shows that a traditional barrier exists even where the origin is irregular.

\begin{theorem}[Traditional barriers for $p>2$]\label{thm:traditional-barrier}
  Let $n\ge2$, $p>2$, and $l,\theta>0$.  Then the origin admits a traditional
  barrier with respect to $\Theta_{l,\theta}$.
\end{theorem}

This answers \cite[Open Problem~1.9(b)]{BB26}.  Combining the theorem with the irregular cases in \cref{thm:classification} answers Lindqvist's question in the degenerate range.

\begin{corollary}[One barrier at an irregular point]\label{cor:one-barrier}
  If
  \begin{equation*}
    2<p<n,
    \qquad 0<l<p,
    \qquad \theta>0,
  \end{equation*}
  then the origin is an irregular boundary point of $\Theta_{l,\theta}$
  that admits a traditional barrier.
\end{corollary}

The irregularity is proved in \cref{sec:irregularity}, and the traditional barrier is constructed in \cref{sec:traditional-barrier}.  Taking $l=2$ in the corollary gives the following geometric consequence.

\begin{corollary}[Exterior ball at a north pole]\label{cor:north-pole}
  For every $2<p<n$, there is a bounded domain with an irregular boundary
  point $\xi_0$ that admits a traditional barrier.  An exterior Euclidean
  space-time ball touches the boundary from the past at $\xi_0$, and
  $\xi_0$ is its north pole.
\end{corollary}

The corollary gives a negative answer to the north-pole question posed in~\cite[Section~8, Problem~5]{BBGP15} in the range~$2<p<n$.  For the heat equation, an exterior space-time ball touching the boundary from the past at its north pole does imply regularity, as a consequence of the tusk condition of Effros--Kazdan~\cite{EK71}.  
For the standard $p$-parabolic equation, the exterior-ball conditions in the literature exclude the north pole touching.

\subsection{Background}

Wiener characterized regularity for the Laplace equation by a capacity condition~\cite{Wiener24}.  For the elliptic $p$-Laplacian, Kilpel\"ainen--Mal\'y~\cite{KM94} established the necessity of the nonlinear Wiener criterion for $1<p\le n$, completing the characterization in that range.  For the heat equation, the corresponding necessity and sufficiency were proved by Lanconelli~\cite{Lan73} and Evans--Gariepy~\cite{EG82}.  For the $p$-parabolic equation with $p\ne2$, a necessary and sufficient Wiener criterion in general space-time domains remains open.  In cylinders, lateral boundary regularity is equivalent to elliptic $p$-harmonic regularity of the spatial boundary point, as shown by Kilpel\"ainen--Lindqvist~\cite{KL96} and~\cite[Theorem~3.9]{BBGP15}.  Quantitative Wiener-type estimates are also known in cylindrical domains for equations of $p$-Laplace type.  Gianazza--Liao--Lukkari~\cite{GLL18} treat $p_*<p<2$, Gianazza--Liao~\cite{GL20} treat $p>2$, and~\cite{GL22} treats $1<p\le p_*$.  These estimates use elliptic $p$-capacity, with a nonoptimal sufficient capacity exponent in the last range.

Several notions of nonlinear parabolic capacity have been developed for the $p$-Laplacian.  A functional approach was introduced in~\cite{DPP03}.  In the degenerate range, parabolic capacity was related to polar sets in~\cite{KKKP13}, and an equivalent variational capacity with estimates for explicit sets was developed in~\cite{AKP15}.  A variational definition valid for all $1<p<\infty$ was given in~\cite{MS25}.  These works address capacity, polar sets, and removability, but do not classify regularity at the moving inner boundary in~\eqref{eq:soda-domain}.  Related boundary estimates include Carleson estimates in cylinders and time-dependent domains, boundary H\"older estimates, and intrinsic Harnack chains~\cite{AGS16,A16,AKN19}.  They concern nonnegative solutions vanishing on a boundary portion.  We ask when Perron solutions attain arbitrary continuous data.

Our approach uses Perron solutions and barrier families in the formulation of~\cite{BBGP15}, following Perron's method for the Laplace equation~\cite{Perron23}, its extension to the heat equation by Sternberg~\cite{Sternberg29}, and the development of potential theory by Granlund--Lindqvist--Martio~\cite{GLM86} and, in the parabolic case, by Kilpel\"ainen--Lindqvist~\cite{KL96}.  Boundary regularity is local and independent of the future~\cite[Section~3]{BBGP15}.  The geometry of moving boundaries is classically tied to Petrovski\u\i's criterion for the heat equation~\cite{Petrovskii35}.  Lindqvist~\cite{L95} and Bj\"orn--Bj\"orn--Gianazza--Parviainen~\cite[Sections~6--7]{BBGP15} developed nonlinear counterparts using Barenblatt-type constructions.  Bj\"orn--Bj\"orn--Gianazza~\cite[Theorem~1.1]{BBG17} classified regularity for the power domains $|x|<K(-t)^{\nicefrac{1}{l}}$ in almost all cases.  There $l<p$ gives regularity and $l>p$ irregularity, while at $l=p$ the origin is irregular for $p>2$ and regular for $p=2$.  The singular case remained open.  The soda-can domain has the complementary geometry and lies outside the moving sphere.  For the heat equation, related rotationally symmetric domains were studied by Abdulla~\cite{Abdulla05}.

The normalized $p$-parabolic equation is another useful comparison because its spatial operator is homogeneous of degree one in the dependent variable.  Banerjee--Garofalo~\cite{BG15} characterize regularity by a single barrier, and Bj\"orn--Bj\"orn--Parviainen~\cite{BBP19} prove a tusk condition and a sharp Petrovski\u\i -criterion. It also follows that in this case, the north pole touching can be included into exterior-ball condition.   
 Kurkinen~\cite[Theorems~5.5 and~6.2]{K24} extends the barrier-family characterization and the singular counterexample to a class containing both the normalized and usual $p$-parabolic equations.  This comparison highlights the role of amplitude scaling in the single-barrier question.

\subsection{Ideas of the proofs}
\label{sec:idea}

The proofs construct smooth radial supersolutions whose spatial gradients do not vanish near the origin. This allows us to check the supersolution inequality pointwise. For barrier supersolutions with positive time derivative, multiplication by a factor at most $1$ preserves the supersolution inequality when $p>2$, while multiplication by a factor at least $1$ preserves it when $p<2$. In the singular range, we multiply local profiles by large constants and use the pasting lemma to join them to constants outside small neighborhoods. This gives barrier families. In the degenerate range, small multiples give a traditional barrier, but do not provide the arbitrarily large boundary values required of a barrier family, so additional ideas are needed.

\paragraph{Irregularity and a traditional barrier.} For $l<p$, the function $$\Phi(x,t)=\left(\frac{(-t)^{\nicefrac{1}{l}}}{|x|}\right)^\gamma$$ from~\cite{BB26} is at most $1$ in the domain, equals $1$ on the initial boundary, and tends to zero along the curve $-t=|x|^p$. When $1<p<n$ and $p\ne2$, adding a suitable function of $t$ to $\Phi$ makes it a supersolution. The comparison principle then shows that the lower Perron solution does not attain the boundary value $1$ at the origin, proving irregularity. 

For $p>2$, we construct a traditional barrier of the form $$w=C|x|^\alpha(1-\eta y),
$$
where $y=(-t)/|x|^l$ and $0<\eta<1$. Its time derivative is proportional to $C$, while its $p$-Laplacian is proportional to $C^{p-1}$. Since $p-1>1$, choosing $C$ small makes the time derivative dominate. A similar comparison of the time derivative with the $p$-Laplacian appears in the exterior-ball barrier of Kilpel\"ainen--Lindqvist, as discussed in~\cite[discussion after Proposition~1.3]{BB26}.

\paragraph{The degenerate cases.} For $l=p$ and $2<p\le n$, we use
\begin{equation*}
  w_\alpha(x,t)=\frac{C_\alpha}{\alpha}
  \left(|x|^\alpha-\eta(-t)^{\nicefrac{\alpha}{p}}\right),
  \qquad 0<\eta<1.
\end{equation*}
For each small $\alpha>0$, we choose $C_\alpha$ so that $w_\alpha$ is a supersolution. It tends to zero at the origin, while $C_\alpha/\alpha\to\infty$ as $\alpha\downarrow0$ makes its boundary values arbitrarily large away from the origin.

When $p=n>2$ and $l<p$, set $\rho=\log(1/|x|)$ and $y=l\rho / \log(1/(-t))$. For a small fixed $\mu>0$, the local profile is
\begin{equation*}
  h_j(x,t)=j\rho^{-\mu}\left(1-\frac y4-\frac{y^2}4\right).
\end{equation*}
Since $n=p$, the first-order term cancels when the radial $p$-Laplacian is written in $\rho$. The quadratic factor in $h_j$ has second derivative $-1/2$ with respect to $y$. When $y$ stays away from zero and $\mu$ is small, this gives $\Deltap h_j\le0$. Near $t=0$, the positive time derivative dominates any positive $\Deltap h_j$. For large $j$, this is a supersolution on $\rho>j$. Truncating $h_j$ at a constant and pasting across $\rho=j$ gives a barrier family.

\paragraph{The singular endpoint.} Set $\rho=\log(1/|x|)$ and $y=(-t)/|x|^p$. For $l=p$ and $p_*<p<2$, we use
\begin{equation*}
  h(x,t)=\rho^{-a}
  \left(1-\rho^{-d_1}y^{\beta_1}-\rho^{-d_2}y^{\beta_2}\right),
  \qquad 0<\beta_1<1,
  \quad \beta_1<\frac{n-p}{p(p-1)}<\beta_2<\frac1{2-p}.
\end{equation*}
The $\beta_2$ correction makes $\Deltap h\le0$ unless $y$ is small. There the $\beta_1$ correction gives a time derivative that dominates any positive $\Deltap h$.

At $l=p=p_*$, the two bounds on $\beta_2$ both equal $(n+1)/2$. With small $\varepsilon>0$, we use
\begin{equation*}
  h(x,t)=(\log\rho)^{-\nicefrac{1}{4}}
  \left[1-\varepsilon\rho^{-\nicefrac{1}{2}}y^{\nicefrac{1}{2}}
  -\varepsilon y^{\frac{n+1}{2}}
  \left(1+\frac{1+\log(1/y)}{\sqrt{\log\rho}}\right)^2\right].
\end{equation*}
The $y^{\nicefrac{1}{2}}$ term supplies $h_t$ near $t=0$. The exponent
$(n+1)/2$ makes the leading spatial contribution of the second correction
vanish. The logarithmic factor in that correction controls the remaining
spatial error. In both cases, multiplication and pasting give a barrier family.

\paragraph{Organization.}  \Cref{sec:preliminaries} recalls Perron solutions, barriers, comparison, pasting, and the radial formula.  \Cref{sec:single-barrier} proves irregularity for $l<p$ and constructs the traditional barrier, which gives \cref{cor:one-barrier,cor:north-pole}.  \Cref{sec:degenerate} treats the degenerate endpoint and the critical dimension.  \Cref{sec:singular-endpoint} treats the singular endpoint and ends with the proof of \cref{thm:classification}.  The technical estimates for that section are collected in \cref{app:section5-estimates}.

\section{Preliminaries}\label{sec:preliminaries}

The constructions below produce smooth supersolutions, while boundary regularity is expressed through Perron solutions and $p$-superparabolic barriers. We recall the passage between these notions, the comparison and pasting principles.

Let $\Theta\subset\R^{n+1}$ be bounded, nonempty, and open, and let
$1<p<\infty$.  For an open set
$U\subset\R^n$ and $t_1<t_2$, write
$U_{t_1,t_2}=U\times(t_1,t_2)$.

\begin{definition}[Weak solutions and $p$-parabolic functions]
  \label{def:weak-solutions}
  A function $u:\Theta\to[-\infty,\infty]$ is a weak solution of
  \eqref{eq:p-parabolic} if, whenever
  $U_{t_1,t_2}\Subset\Theta$ is an open cylinder,
  \begin{equation*}
    u\in C\bigl((t_1,t_2);L^2(U)\bigr)
    \cap L^p\bigl((t_1,t_2);W^{1,p}(U)\bigr)
  \end{equation*}
  and
  \begin{equation}\label{eq:weak-formulation}
    \int_{t_1}^{t_2}\!\int_U
    |\nabla u|^{p-2}\nabla u\cdot\nabla\varphi\,dx\,dt
    -\int_{t_1}^{t_2}\!\int_U
    u\,\partial_t\varphi\,dx\,dt=0
  \end{equation}
  for every $\varphi\in C_0^\infty(U_{t_1,t_2})$.  A continuous weak
  solution is called a $p$-parabolic function.

  A function $u$ is a weak supersolution, respectively weak subsolution,
  if $u\in L^p\bigl((t_1,t_2);W^{1,p}(U)\bigr)$ for every such cylinder and the
  left-hand side of \eqref{eq:weak-formulation} is nonnegative,
  respectively nonpositive, for every nonnegative
  $\varphi\in C_0^\infty(U_{t_1,t_2})$.
\end{definition}

As in \cite[Section~2]{BBGP15}, we often omit ``weak'' when referring to
weak solutions, supersolutions, and subsolutions.

For a box $Q=(a_1,b_1)\times\cdots\times(a_n,b_n)\subset\R^n$, the set
$Q_{t_1,t_2}=Q\times(t_1,t_2)$ is called a space-time box, and its
parabolic boundary is
\begin{equation*}
  \partial_pQ_{t_1,t_2}
  =\bigl(\overline Q\times\{t_1\}\bigr)
  \cup\bigl(\partial Q\times(t_1,t_2]\bigr).
\end{equation*}

The comparison definition below follows Kilpel\"ainen--Lindqvist~\cite[Definition~4.1]{KL96} and Korte--Kuusi--Parviainen~\cite[Definition~2.6]{KKP10}.
\begin{definition}[$p$-superparabolic and $p$-subparabolic functions]
  \label{def:superparabolic}
  A function $u:\Theta\to(-\infty,\infty]$ is $p$-superparabolic if
    \begin{enumerate}[label=(\roman*)]
        \item $u$ is lower semicontinuous;
        \item $u$ is finite in a dense subset of $\Theta$;
        \item $u$ satisfies the following comparison principle on each
              space-time box $Q_{t_1,t_2}\Subset\Theta$: if $h$ is $p$-parabolic in
              $Q_{t_1,t_2}$ and continuous on $\overline{Q_{t_1,t_2}}$, and if
              $h\le u$ on $\partial_pQ_{t_1,t_2}$, then $h\le u$ throughout
              $Q_{t_1,t_2}$.
      \end{enumerate}
    A function $v:\Theta\to[-\infty,\infty)$ is
  $p$-subparabolic if $-v$ is $p$-superparabolic.
\end{definition}

For bounded $f:\partial\Theta\to\R$, let the upper class $\mathcal U_f$
consist of the $p$-superparabolic functions $u$ in $\Theta$, bounded from
below, such that
\begin{equation*}
  \liminf_{\Theta\ni\eta\to\xi}u(\eta)\ge f(\xi)
  \qquad\text{for every }\xi\in\partial\Theta.
\end{equation*}
The upper Perron solution of $f$ is
\begin{equation*}
  \upperH_\Theta f(\xi)=\inf_{u\in\mathcal U_f}u(\xi),
  \qquad \xi\in\Theta.
\end{equation*}
Similarly, let the lower class $\mathcal L_f$ consist of the
$p$-subparabolic functions $v$ in $\Theta$, bounded from above, such that
\begin{equation*}
  \limsup_{\Theta\ni\eta\to\xi}v(\eta)\le f(\xi)
  \qquad\text{for every }\xi\in\partial\Theta.
\end{equation*}
The lower Perron solution of $f$ is
\begin{equation*}
  \lowerH_\Theta f(\xi)=\sup_{v\in\mathcal L_f}v(\xi),
  \qquad \xi\in\Theta.
\end{equation*}
Moreover,
\begin{equation}\label{eq:upper-lower}
  \lowerH_\Theta f=-\upperH_\Theta(-f).
\end{equation}
A boundary point $\xi_0\in\partial\Theta$ is regular with respect to
$\Theta$ if
\begin{equation*}
  \lim_{\Theta\ni\xi\to\xi_0}\upperH_\Theta f(\xi)=f(\xi_0)
\end{equation*}
for every $f\in C(\partial\Theta)$.
By \eqref{eq:upper-lower}, lower Perron solutions give the same definition.

A positive $p$-superparabolic function $w$ in $\Theta$ is a
traditional barrier in $\Theta$ at $\xi_0\in\partial\Theta$ if
\begin{equation*}
  \lim_{\Theta\ni\zeta\to\xi_0}w(\zeta)=0
\end{equation*}
and
\begin{equation*}
  \liminf_{\Theta\ni\zeta\to\xi}w(\zeta)>0
  \qquad\text{for every }\xi\in\partial\Theta\setminus\{\xi_0\}.
\end{equation*}
We use the barrier-family notion from \cite[Definition~3.2]{BBGP15}.
\begin{definition}[Barrier family]\label{def:barrier-family}
  Let $\xi_0\in\partial\Theta$.  A family of functions
  $w_j:\Theta\to(0,\infty]$, $j=1,2,\ldots$, is a barrier family in
  $\Theta$ at $\xi_0$ if
  \begin{enumerate}[label=(\arabic*)]
    \item $w_j$ is a positive $p$-superparabolic function in $\Theta$ for
          every $j$;
    \item
          \begin{equation*}
            \lim_{\Theta\ni\zeta\to\xi_0}w_j(\zeta)=0
            \qquad\text{for every }j;
          \end{equation*}
    \item for each integer $k\ge1$, there is a $j$ such that
          \begin{equation}\label{eq:barrier-family}
            \liminf_{\Theta\ni\zeta\to\xi}w_j(\zeta)\ge k
            \quad\text{for all }\xi\in\partial\Theta
            \text{ with }|\xi-\xi_0|\ge\frac1k.
          \end{equation}
  \end{enumerate}
\end{definition}
By \cite[Theorem~3.3]{BBGP15}, a boundary point is regular if and only if it
has a barrier family.

We use the following parabolic comparison principle
\cite[Theorem~2.4]{BBGP15}.

\begin{proposition}[Parabolic comparison principle]\label{prop:comparison}
  Let $u$ be $p$-superparabolic and $v$ be $p$-subparabolic in a
  bounded open set $\Theta$.  Let $T\in\R$, and assume that
  \begin{equation*}
    +\infty\ne\limsup_{\Theta\ni(y,s)\to(x,t)}v(y,s)
    \le \liminf_{\Theta\ni(y,s)\to(x,t)}u(y,s)\ne-\infty
  \end{equation*}
  at every $(x,t)\in\partial\Theta$ with $t<T$.  Then
  \begin{equation*}
    v\le u\qquad\text{in }\Theta\cap\{t<T\}.
  \end{equation*}
\end{proposition}

In \cref{sec:irregularity} we use this with $T=0$ to compare a lower Perron competitor with a supersolution without requiring an inequality on the top $t=0$.

We also use the pasting lemma \cite[Lemma~2.9]{BBGP15}.  If $G\subset
  \Theta$ is open, $u$ is $p$-superparabolic in $\Theta$, and $v$
is $p$-superparabolic in $G$, then
\begin{equation*}
  w=\begin{cases}
    \min\{u,v\}, & \text{in }G,                 \\
    u,           & \text{in }\Theta\setminus G,
  \end{cases}
\end{equation*}
is $p$-superparabolic whenever it is lower semicontinuous.  All local
supersolutions pasted below are continuous across the pasting interface.
Continuous weak supersolutions are $p$-superparabolic by the comparison
principle for weak subsolutions and supersolutions; see
\cite[Lemma~3.1]{KL96} and \cite[Lemma~3.5]{KKP10}.

Finally, for a radial function
$u(x)=V(|x|)$, set $r=|x|$.  Whenever $V'(r)\ne0$, we have
\begin{equation}\label{eq:radial-formula}
  \Deltap u(x)
  =|V'(r)|^{p-2}
  \left((p-1)V''(r)+\frac{n-1}{r}V'(r)\right).
\end{equation}
This equality will be used throughout the paper as all barriers are rotationally symmetric.

\section{A barrier at an irregular point}\label{sec:single-barrier}

This section shows that the origin can have a traditional barrier and still be irregular. We first prove irregularity by comparison, then construct the barrier at the same point.

\subsection{Irregularity when $l<p$, $1<p<n$ }\label{sec:irregularity}

\begin{theorem}[Irregularity when $l<p$]\label{thm:irregularity-below}
  Let $n\ge2$, $1<p<n$, $p\ne2$, and $0<l<p$.  Then the origin is
  irregular with respect to $\Theta_{l,\theta}$ for every $\theta>0$.
\end{theorem}

Consider $\theta=1$ and the profile
\begin{equation*}
  \Phi(x,t)=\left(\frac{(-t)^{\nicefrac{1}{l}}}{|x|}\right)^\gamma,
  \qquad \gamma>0,
\end{equation*}
which satisfies $0<\Phi<1$ in the domain. For $1<p<2$ and $l<p$ with
$l<(n-p)(2-p)/(p-1)$, a multiple of this profile is a supersolution
when $\gamma=l/(2-p)$, as in~\cite{BB26}. To cover all $l<p$, we
choose $\gamma$ small. Then $\Phi_t-\Deltap\Phi<0$ near $t=0$ at fixed
$x$, so we use
\begin{equation*}
  u(x,t)=\Phi(x,t)-L(-t)^\delta,
  \qquad \delta,L>0.
\end{equation*}
The correction leaves $\Deltap\Phi$ unchanged and adds
$L\delta(-t)^{\delta-1}$ to the time derivative. We choose $\delta$ to
balance the time powers in $u_t-\Deltap u$, and then choose $L$ large
enough to make this expression nonnegative. The correction still
vanishes as $t\uparrow0$.
\Cref{fig:irregularity-comparison} illustrates $u$ and its limits along
two approaches to the origin.

\begin{figure}[!htbp]
  \centering
  \begin{tikzpicture}
  \definecolor{boundaryblue}{RGB}{0,80,150}
  \definecolor{pathorange}{RGB}{180,75,0}
  \begin{axis}[
    name=surface, width=8.2cm, height=6.4cm,
    view={35}{25},
    xlabel={$t$}, zlabel={$u$},
    xmin=-0.25, xmax=0, ymin=0, ymax=0.25, zmin=-0.3, zmax=1.08,
    xtick={-0.2,-0.1,0}, ytick={0,0.1,0.2}, ztick={0,0.5,1},
    tick label style={font=\footnotesize},
    label style={font=\small},
    title={\small The comparison function near the origin},
    axis lines=box, grid=major, grid style={black!8},
    colormap={surfacegray}{color(0cm)=(black!6); color(1cm)=(black!22)},
    clip=false,
  ]
    \addplot3[surf, shader=interp, domain=0:1, y domain=0:1,
      samples=45, samples y=31, z buffer=sort]
      ({-(0.0004+0.2496*x^2)*y}, {0.0004+0.2496*x^2},
       {y-2*sqrt((0.0004+0.2496*x^2)*y)});
    \addplot3[boundaryblue, very thick, domain=0:0.5, samples=101,
      samples y=0] ({-x^2},{x^2},{1-2*x});
    \addplot3[pathorange, very thick, dashed, domain=0:0.5, samples=101,
      samples y=0] ({-x^6},{x^2},{x^4-2*x^3});
    \addplot3[only marks, mark=o, mark size=2.5pt,
      mark options={solid,fill=white,draw=boundaryblue,thick}]
      coordinates {(0,0,1)};
    \addplot3[only marks, mark=o, mark size=2.5pt,
      mark options={solid,fill=white,draw=pathorange,thick}]
      coordinates {(0,0,0)};
    \node[font=\small, sloped like y axis, yshift=-20pt]
      at (axis cs:0,0.125,-0.3) {$r=|x|$};
  \end{axis}
  \begin{axis}[
    at={(8.7cm,0.55cm)}, anchor=south west,
    width=6.2cm, height=5.45cm,
    xlabel={$r$}, ylabel={$u$},
    xmin=0, xmax=0.25, ymin=-0.3, ymax=1.08,
    xtick={0,0.1,0.2}, ytick={0,0.5,1},
    tick label style={font=\footnotesize}, label style={font=\small},
    title={\small Values along the two approaches},
    axis lines=left, grid=major, grid style={black!8},
    legend style={font=\footnotesize,draw=none,fill=none,
      at={(0.98,0.98)},anchor=north east,cells={anchor=west}},
    clip=false,
  ]
    \addplot[boundaryblue, very thick, domain=0:0.5, samples=101]
      ({x^2},{1-2*x});
    \addlegendentry{$t=-r$: $1-2\sqrt r$}
    \addplot[pathorange, very thick, dashed, domain=0:0.5, samples=101]
      ({x^2},{x^4-2*x^3});
    \addlegendentry{$t=-r^3$: $r^2-2r^{3/2}$}
    \addplot[only marks, mark=o, mark size=2.5pt,
      mark options={solid,fill=white,draw=boundaryblue,thick},forget plot]
      coordinates {(0,1)};
    \addplot[only marks, mark=o, mark size=2.5pt,
      mark options={solid,fill=white,draw=pathorange,thick},forget plot]
      coordinates {(0,0)};
  \end{axis}
\end{tikzpicture}
  \caption{The comparison function in the proof of
    \cref{thm:irregularity-below}, with $n=6$, $p=3$, $l=\gamma=1$,
    $\delta=1/2$, and $L=2$. Writing $r=|x|$, we have
    $u(x,t)=(-t)/r-2\sqrt{-t}$ on $0<-t<r<1$.
    The solid curve is its trace on the initial boundary $t=-r$.
    The dashed curve follows the interior path $t=-r^3$.
    The open circles mark the respective limits $1$ and $0$ as
    $(x,t)\to(0,0)$. Since $\lowerH_{\Theta_{1,1}}f\le u$
    and $f(0,0)=1$, the interior path proves irregularity.}
  \label{fig:irregularity-comparison}
\end{figure}
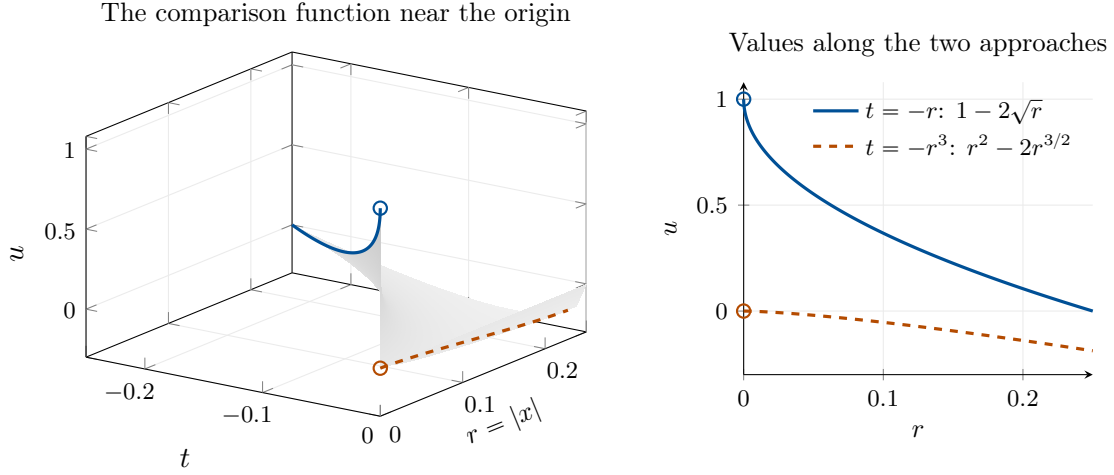

\begin{proof}
  Here and later, we may assume $\theta=1$ by
  \cite[Lemma~3.6]{BB26}.  Since $p<n$, we can choose $\gamma>0$
  small enough that
  \begin{equation*}
    A_\gamma:=\gamma^{p-1}\bigl(n-p-(p-1)\gamma\bigr)>0,
    \qquad
    \nu:=p-1+\frac p\gamma>1.
  \end{equation*}
  With this choice of $\gamma$, let $\Phi$ and $u$ be as above,
  with $\delta,L>0$ to be chosen below.
  Since $x\ne0$ and $\nabla\Phi(x,t)\ne0$ in $\Theta_{l,1}$,
  \eqref{eq:radial-formula} gives
  \begin{align*}
    \Deltap u(x,t)=\Deltap\Phi(x,t)
     & =-A_\gamma(-t)^{\frac{\gamma(p-1)}{l}}|x|^{-\gamma(p-1)-p} \\
     & =-A_\gamma(-t)^{-\nicefrac{p}{l}}\Phi(x,t)^\nu.
  \end{align*}
  Differentiating with respect to $t$, we obtain
  \begin{equation*}
    u_t-\Deltap u
    =L\delta(-t)^{\delta-1}
    -\frac{\gamma}{l}(-t)^{-1}\Phi
    +A_\gamma(-t)^{-\nicefrac{p}{l}}\Phi^\nu.
  \end{equation*}
  Choose
  \begin{equation*}
    \delta=\frac{p/l-1}{\nu-1}>0,
    \qquad\text{so that}\qquad
    1-\delta-\frac pl=-\delta\nu,
  \end{equation*}
  and put $y=(-t)^{-\delta}\Phi(x,t)>0$.  Then
  \begin{equation}\label{eq:irregularity-residual}
    u_t(x,t)-\Deltap u(x,t)
    =(-t)^{\delta-1}
    \left(L\delta-\frac{\gamma}{l}y+A_\gamma y^\nu\right).
  \end{equation}
  Since $A_\gamma>0$ and $\nu>1$, the function
  $y\mapsto A_\gamma y^\nu-(\gamma/l)y$ is bounded from below on
  $[0,\infty)$.  Thus, for sufficiently large $L>0$,
  \eqref{eq:irregularity-residual} gives $u_t-\Deltap u\ge0$ in
  $\Theta_{l,1}$.  As $u$ is smooth and $\nabla u=\nabla\Phi\ne0$
  there, it is a continuous weak supersolution and hence $p$-superparabolic.

  Let $u$ also denote its continuous boundary trace on
  $\partial\Theta_{l,1}\cap\{t<0\}$, and define
  $f:\partial\Theta_{l,1}\to\R$ by
  \begin{equation*}
    f(x,t)=
    \begin{cases}
      u(x,t), & (x,t)\in\partial\Theta_{l,1},\ t<0, \\[1mm]
      1-|x|,  & (x,t)\in\partial\Theta_{l,1},\ t=0.
    \end{cases}
  \end{equation*}
  Along the initial boundary,
  \begin{equation*}
    u(x,-|x|^l)=1-L|x|^{l\delta}\to 1=f(0,0)
    \qquad\text{as }x\to0,
  \end{equation*}
  while along the lateral boundary,
  \begin{equation*}
    u(x,t)=(-t)^{\nicefrac{\gamma}{l}}-L(-t)^\delta\to 0=1-|x|
    \qquad\text{as }t\to0^-\text{ with }|x|=1.
  \end{equation*}
  Hence $f\in C(\partial\Theta_{l,1})$ and $f(0,0)=1$.
  At every boundary point with $t<0$, the boundary limit of $u$ equals $f$.
  If $v\in\mathcal L_f$, then
  \cref{prop:comparison}, with terminal time $T=0$, gives $v\le u$.
  Taking the supremum over the lower class yields
  \begin{equation}
    \label{eq:comp}
    \lowerH_{\Theta_{l,1}}f\le u.
  \end{equation}

  Along the path $((-t)^{\nicefrac{1}{p}}e_1,t)$, for $t \in (-1,0)$, we have
  $((-t)^{\nicefrac{1}{p}}e_1,t)\in\Theta_{l,1}$ because $l<p$, and
  \begin{equation*}
    u((-t)^{\nicefrac{1}{p}}e_1,t)
    =(-t)^{\gamma(\nicefrac{1}{l}-\nicefrac{1}{p})}-L(-t)^\delta\to 0
    \qquad\text{as }t\to0^-.
  \end{equation*}
  Combining this with \eqref{eq:comp} gives
  \begin{equation*}
    \limsup_{t\to0^-}
    \lowerH_{\Theta_{l,1}}f((-t)^{\nicefrac{1}{p}}e_1,t)
    \le0<f(0,0).
  \end{equation*}
  Thus the lower Perron solution does not attain the prescribed boundary
  value at the origin, so the origin is irregular.
\end{proof}

\begin{remark}
  The new cases in \cref{thm:irregularity-below}, beyond
  \cite[Proposition~1.4]{BB26}, are $p_*<p<2$ with
  $(n-p)(2-p)/(p-1)\le l<p$ and $2<p<n$ with $0<l<p$.
\end{remark}

\subsection{One traditional barrier for every $0<l$, $2<p$}\label{sec:traditional-barrier}

We recall that the idea of the proof is given in Section~\ref{sec:idea}.

\begin{proof}[Proof of \cref{thm:traditional-barrier}]
  Choose $0<\eta<1$ and
  \begin{equation*}
    \alpha>\max\left\{l,\frac{p-l}{p-2}\right\}.
  \end{equation*}
  For $C>0$ to be chosen below, set
  \begin{equation*}
    w(x,t)=C|x|^\alpha\left(1-\eta\frac{-t}{|x|^l}\right).
  \end{equation*}
  Since $0<(-t)/|x|^l<1$ in $\Theta_{l,1}$,
  \begin{equation}\label{eq:traditional-two-sided}
    C(1-\eta)|x|^\alpha<w(x,t)<C|x|^\alpha.
  \end{equation}

  Write $r=|x|$ and $y=(-t)/r^l$. Differentiating gives
  \begin{equation*}
    w_r=Cr^{\alpha-1}\bigl[\alpha-\eta(\alpha-l)y\bigr],
    \qquad
    w_{rr}=Cr^{\alpha-2}
    \bigl[\alpha(\alpha-1)-\eta(\alpha-l)(\alpha-l-1)y\bigr].
  \end{equation*}
  Since $0<y<1$ and $\alpha>l$, these formulas imply
  $0<w_r\le\alpha Cr^{\alpha-1}$ and
  $|w_{rr}|\le KCr^{\alpha-2}$, where $K$ depends only on
  $\alpha,l,\eta$.  By \eqref{eq:radial-formula},
  \begin{equation*}
    \Deltap w
    =w_r^{p-2}\left((p-1)w_{rr}+\frac{n-1}{r}w_r\right)
    \le MC^{p-1}r^{\alpha(p-1)-p},
  \end{equation*}
  where $M>0$ depends only on $n,p,l,\alpha,\eta$.
  Since $w_t=C\eta r^{\alpha-l}$, it follows that
  \begin{equation*}
    w_t-\Deltap w
    \ge Cr^{\alpha-l}
    \left[\eta-MC^{p-2}r^{\alpha(p-2)+l-p}\right].
  \end{equation*}
  Our choice of $\alpha$ gives $\alpha(p-2)+l-p>0$.
  Thus, since $0<r<1$ and $p>2$, choosing $C>0$ sufficiently small
  that $MC^{p-2}\le\eta$ yields $w_t-\Deltap w\ge0$ in $\Theta_{l,1}$.

  Hence $w$ is a continuous weak supersolution and therefore
  $p$-superparabolic.  The bounds in \eqref{eq:traditional-two-sided}
  show that $w$ tends to zero at the origin.  Every other boundary
  point $(x_0,t_0)$ has $|x_0|>0$, and the lower bound gives a strictly
  positive boundary limit there.  Hence $w$ is a traditional barrier
  at the origin.
\end{proof}

\begin{proof}[Proof of \cref{cor:one-barrier}]
  In the stated range, \cref{thm:irregularity-below} shows that the origin is
  irregular, while \cref{thm:traditional-barrier} provides a traditional
  barrier there.
\end{proof}

\begin{proof}[Proof of \cref{cor:north-pole}]
  Take $l=2$ and $\theta=1$, and write $\Theta=\Theta_{2,1}$.  Consider the
  Euclidean ball in $\R^{n+1}$ with centre $(0,-1/2)$ and radius $1/2$.
  Every point in its closure satisfies
  \begin{equation*}
    |x|^2+(t+1/2)^2\le1/4,
  \end{equation*}
  and hence $-t\ge |x|^2+t^2\ge |x|^2$.  The closed ball is therefore
  contained in the complement of $\Theta$, and the origin is its north
  pole.  Since $2<p$, \cref{cor:one-barrier} shows that the origin is
  irregular while it admits a traditional barrier.
\end{proof}

\begin{remark}\label{rem:exact-degenerate-barrier}
  In the range $p/(p-1)<l<p$, $p>2$, the exact $p$-parabolic function
  \begin{equation*}
    \frac{\kappa(p-1)}p|x|^{\frac{p}{p-1}}+n\kappa^{p-1}t
  \end{equation*}
  is another traditional barrier when $\kappa>0$ is sufficiently small.

\end{remark}

\section{Regularity when $p>2$}\label{sec:degenerate}

\subsection{Regularity for $l=p$, $2<p\le n$}

When $p>2$, multiplying the preceding barrier by a large constant need
not give a supersolution.  At $l=p$, we instead vary its radial exponent.

\begin{theorem}[Regularity at the endpoint $l=p$]\label{thm:degenerate-endpoint}
  Let $2<p\le n$ and $\theta>0$.  Then the origin is regular with respect
  to $\Theta_{p,\theta}$.
\end{theorem}

\begin{proof}
  By \cite[Lemma~3.6]{BB26}, it suffices to prove the claim for
  $\Theta_{p,1}$.  We start by constructing a $p$-superparabolic function
  in $\Theta_{p,1}$.  Let $\alpha>0$ and let $C_\alpha>0$.
  Fix $0<\eta<1$ independently of $\alpha$ and define
  \begin{equation*}
    w_\alpha(x,t)
    =\frac{C_\alpha}{\alpha}
    \left(|x|^\alpha-\eta(-t)^{\nicefrac{\alpha}{p}}\right).
  \end{equation*}
  Since $(-t)^{\nicefrac{\alpha}{p}}<|x|^\alpha$ in $\Theta_{p,1}$, the subtracted term
  is less than $\eta|x|^\alpha$.  Thus
  \begin{equation}\label{eq:endpoint-two-sided}
    (1-\eta)\frac{C_\alpha}{\alpha}|x|^\alpha
    <w_\alpha(x,t)
    <\frac{C_\alpha}{\alpha}|x|^\alpha.
  \end{equation}
  Hence, $w_\alpha$ is positive in $\Theta_{p,1}$ and tends to zero at the origin.

  Fix $t<0$ and set $r=|x|$.  With a slight
  abuse of notation, write $w_\alpha(r,t)$ for the radial profile.  Then
  \begin{equation*}
    (w_\alpha)_r(r,t)=C_\alpha r^{\alpha-1},
    \qquad
    (w_\alpha)_{rr}(r,t)=C_\alpha(\alpha-1)r^{\alpha-2}.
  \end{equation*}
  Consequently,
  \begin{equation*}
    (p-1)(w_\alpha)_{rr}(r,t)
    +\frac{n-1}{r}(w_\alpha)_r(r,t)
    =C_\alpha D_\alpha r^{\alpha-2},
    \qquad
    D_\alpha:=n-p+(p-1)\alpha>0,
  \end{equation*}
  where the positivity follows from $n\ge p$ and $\alpha>0$.
  The radial formula \eqref{eq:radial-formula} therefore gives
  \begin{equation*}
    \Deltap w_\alpha(x,t)
    =C_\alpha^{p-1}D_\alpha |x|^{\alpha(p-1)-p},
  \end{equation*}
  whereas
  \begin{equation*}
    (w_\alpha)_t(x,t)
    =\frac{\eta C_\alpha}{p}(-t)^{\nicefrac{\alpha}{p}-1}.
  \end{equation*}
  Hence
  \begin{align*}
    (w_\alpha)_t(x,t)-\Deltap w_\alpha(x,t)
    ={} & C_\alpha |x|^{\alpha-p}
    \Biggl[
    \frac{\eta}{p}
    \left(\frac{|x|^p}{-t}\right)^{1-\nicefrac{\alpha}{p}}
    -C_\alpha^{p-2}D_\alpha |x|^{\alpha(p-2)}
    \Biggr].
  \end{align*}
  By \eqref{eq:soda-domain}, $|x|^p/(-t)>1$ and $0<|x|<1$ in
  $\Theta_{p,1}$.  Thus, if we also require $\alpha<p$, then
  $(|x|^p/(-t))^{1-\nicefrac{\alpha}{p}}>1$ and $|x|^{\alpha(p-2)}<1$ since $p>2$.
  We balance the coefficients by choosing $C_\alpha$ such that
  \begin{equation*}
    C_\alpha^{p-2}D_\alpha=\frac{\eta}{p}.
  \end{equation*}
  Consequently,
  \begin{equation*}
    (w_\alpha)_t(x,t)-\Deltap w_\alpha(x,t)
    =\frac{\eta C_\alpha}{p}|x|^{\alpha-p}
    \left[
    \left(\frac{|x|^p}{-t}\right)^{1-\nicefrac{\alpha}{p}}
    -|x|^{\alpha(p-2)}
    \right]>0.
  \end{equation*}
  Moreover, $\nabla w_\alpha(x,t)=C_\alpha|x|^{\alpha-2}x\ne0$ in
  $\Theta_{p,1}$.  Thus the pointwise inequality makes $w_\alpha$ a
  continuous weak supersolution, and hence $p$-superparabolic, in
  $\Theta_{p,1}$.

  We now verify that varying $\alpha$ gives a barrier family in
  $\Theta_{p,1}$ at the origin; see \cref{def:barrier-family}.
  To do so, we first choose $0<\alpha_j<1$ with $\alpha_j\downarrow0$.  Now, since
  $D_\alpha\le n-1$ whenever $0<\alpha\le1$,
  \begin{equation*}
    C_\alpha\ge
    \left(\frac{\eta}{p(n-1)}\right)^{\frac{1}{p-2}}>0.
  \end{equation*}
  It follows, since the lower bound is a constant, that
  \begin{equation*}
    \frac{C_{\alpha_j}}{\alpha_j}\to \infty.
  \end{equation*}
  Fix an integer $k\ge1$.  On $\overline{\Theta}_{p,1}$,
  \begin{equation*}
    |\xi|^2=|x|^2+t^2\le |x|^2+|x|^{2p}\le2|x|^2,
    \qquad \xi=(x,t),
  \end{equation*}
  because $-t\le|x|^p$ and $|x|\le1$.  Therefore every
  $\xi\in\partial\Theta_{p,1}$ with $|\xi|\ge1/k$ satisfies
  \begin{equation*}
    |x|\ge\frac1{k\sqrt2}>0.
  \end{equation*}
  For every such boundary point, \eqref{eq:endpoint-two-sided} bounds the
  lower limit of $w_{\alpha_j}$ from below by
  \begin{equation*}
    (1-\eta)\frac{C_{\alpha_j}}{\alpha_j}
    \left(\frac1{k\sqrt2}\right)^{\alpha_j},
  \end{equation*}
  which tends to infinity because the power factor tends to one while
  $C_{\alpha_j}/\alpha_j\to\infty$.  Thus, for each $k$, we can choose
  $j=j(k)$ so that this lower bound is at least $k$.  Together with
  \eqref{eq:endpoint-two-sided}, this proves that
  $\{w_{\alpha_j}\}_{j=1}^\infty$ is a barrier family in $\Theta_{p,1}$
  at the origin in the sense of \Cref{def:barrier-family}.  The barrier-family
  characterization now implies that the origin is regular in
  $\Theta_{p,1}$.
\end{proof}

\subsection{Regularity for $0<l<p$ when $2<p=n$}\label{sec:critical-dimension}

\begin{theorem}[Regularity below the endpoint in the critical dimension]\label{thm:critical-dimension}
  Let $0<l<p$, $2<p=n$, and $\theta>0$.  Then the origin is regular with respect
  to $\Theta_{l,\theta}$.
\end{theorem}

\begin{proof}
  By \cite[Lemma~3.6]{BB26}, it suffices to work in $\Theta_{l,1}$. Our goal is to create a barrier family to show that the origin is regular. We start by choosing our profile and providing basic formulas.

  \smallskip
  \noindent\emph{Step 1: Derivative formulas and choice of profile.}
  Set $r=|x|$ and
  \begin{equation}\label{eq:critical-dimension-coordinates}
    \rho=\log\frac1r,
    \qquad
    y=\frac{l\rho}{\log(1/(-t))}.
  \end{equation}
  In $\Theta_{l,1}$, the inequalities $0<-t<r^l<1$ give
  $\log(1/(-t))>l\rho>0$ and hence $0<y<1$.
  Let now $F\in C^2([0,1])$ and $A,\mu>0$ to be fixed later.   Put
  \begin{equation*}
    h(x,t)=A\rho^{-\mu}F(y).
  \end{equation*}
  Since
  \begin{equation*}
    y_\rho=\frac{l}{\log(1/(-t))}=\frac y\rho,
  \end{equation*}
  direct differentiation with slight abuse of notation gives
  \begin{align}
    h_\rho
    ={} & A\rho^{-\mu-1}\left[-\mu F(y)+yF'(y)\right],
    \label{eq:critical-first-rho}                      \\
    h_{\rho\rho}
    ={} & A\rho^{-\mu-2}\left[
      \mu(\mu+1)F(y)-2\mu yF'(y)+y^2F''(y)\right].
    \label{eq:critical-second-rho}
  \end{align}
  Moreover,
  \begin{equation*}
    y_t=-\frac{l\rho}{(-t)\bigl(\log(1/(-t))\bigr)^2}
    =-\frac{y^2}{l\rho(-t)}.
  \end{equation*}
  It follows that
  \begin{equation*}
    h_t(x,t)
    =\frac{A\rho^{-\mu}[-F'(y)]y^2}{l\rho(-t)}.
  \end{equation*}
  We also have
  \begin{equation*}
    h_r=-r^{-1}h_\rho,
    \qquad
    h_{rr}=r^{-2}\bigl(h_{\rho\rho}+h_\rho\bigr).
  \end{equation*}
  Since $n=p$,
  \begin{equation*}
    (p-1)h_{rr}+\frac{n-1}{r}h_r
    =(p-1)r^{-2}h_{\rho\rho},
  \end{equation*}
  and $|h_r|^{p-2}=r^{2-p}|h_\rho|^{p-2}$.  Thus, wherever
  $h_\rho\ne0$, \eqref{eq:radial-formula} gives
  \begin{equation}\label{eq:critical-log-plap}
    \Deltap h(x,t)
    =(p-1)r^{-p}|h_\rho|^{p-2}h_{\rho\rho}.
  \end{equation}
  These formulas guide the choice of $F$.  If $F>0$ and $F'<0$ on $[0,1]$,
  then $h>0$, $h_t>0$, and $h_\rho<0$.  Hence
  \eqref{eq:critical-log-plap} applies, and $\Deltap h$ has the sign of
  $h_{\rho\rho}$.  In \eqref{eq:critical-second-rho}, the terms
  $\mu(\mu+1)F$ and $-2\mu yF'$ are positive, while $y^2F''$ is negative
  when $F''<0$.  For $y$ bounded away from zero, we choose $\mu$ small
  enough to make $h_{\rho\rho}<0$.  For small $y$, we instead use a lower
  bound for $-F'$ and the decay of
  $(-t)r^{-p}=\exp[-\rho(l/y-p)]$ to make $h_t$ dominate any positive
  $\Deltap h$ after localization.  We choose the decreasing concave function
  \begin{equation*}
    F(\sigma)=1-\frac\sigma4-\frac{\sigma^2}4.
  \end{equation*}
  For this choice,
  \begin{equation}\label{eq:critical-profile-bounds}
    \frac12\le F(\sigma)\le1,
    \qquad
    -\frac34\le F'(\sigma)\le-\frac14,
    \qquad
    F''(\sigma)=-\frac12
  \end{equation}
  for $0\le\sigma\le1$.

  The preceding time derivative together with
  \eqref{eq:critical-profile-bounds} gives
  \begin{equation}\label{eq:critical-time-lower}
    h_t
    \ge\frac{A\rho^{-\mu}y^2}{4l\rho(-t)}>0.
  \end{equation}

  To show that the localized profiles are $p$-superparabolic, we split
  $\Theta_{l,1}$ into two regions because $\Deltap h$ can be positive when
  $y$ is close to zero.  Let $y_0\in(0,1)$ be the splitting point, to be
  chosen below.  In Region I, where $y_0\le y<1$, we choose $\mu$ so that
  $\Deltap h\le0$.  In Region II, where $0<y<y_0$, we use $h_t>0$ to
  control any positive $\Deltap h$ after localization.  The regions are
  shown in \cref{fig:critical-y-split} for the eventual choice
  $y_0=l/(2p)$.

  \begin{figure}[H]
    \centering
    \begin{tikzpicture}[x=1.6cm,y=1.6cm,>=stealth,
        every node/.style={font=\small}]
      \def\Rmax{4.8}
      \def\rj{1.65}
      \path[fill=black!4]
      (0,0)--(\Rmax,0)--(\Rmax,-1.55)
      -- plot[smooth,domain=\Rmax:0,samples=80]
      (\x,{-1.55*(\x/\Rmax)^6}) -- cycle;
      \path[fill=black!12]
      plot[smooth,domain=0:\Rmax,samples=80]
      (\x,{-1.55*(\x/\Rmax)^6})
      -- plot[smooth,domain=\Rmax:0,samples=80]
      (\x,{-1.55*(\x/\Rmax)^2}) -- cycle;
      \path[fill=black!30,fill opacity=.22]
      (0,0)--(\rj,0)
      -- plot[smooth,domain=\rj:0,samples=30]
      (\x,{-1.55*(\x/\Rmax)^2}) -- cycle;
      \draw[thick] plot[smooth,domain=0:\Rmax,samples=80]
      (\x,{-1.55*(\x/\Rmax)^2});
      \draw[semithick] plot[smooth,domain=0:\Rmax,samples=80]
      (\x,{-1.55*(\x/\Rmax)^6});
      \draw[thick] (\Rmax,-1.55)--(\Rmax,0);
      \draw[densely dashed]
      (\rj,0.30)--(\rj,{-1.55*(\rj/\Rmax)^2});
      \draw[->] (0,-1.78)--(0,0.42) node[above] {$t$};
      \draw[->] (0,0)--(5.22,0) node[right] {$r=|x|$};
      \fill (0,0) circle (1.2pt) node[above left] {$(0,0)$};
      \node[above] at (4.12,0) {$t=0$ (top)};
      \node[above,font=\scriptsize] at (0.72,0.07)
      {local $h_j$, $A=j$};
      \node[above right,font=\scriptsize] at (\rj,0.19)
      {$r=e^{-j}$ $(\rho=j)$};
      \node[font=\small] at (4.00,-0.20) {II};
      \node[font=\small] at (3.80,-0.67) {I};
    \end{tikzpicture}
    \caption{The split in \cref{thm:critical-dimension}: Region I has
    $-r^l<t\le-r^{2p}$; Region II has $-r^{2p}<t<0$.
    The dashed line marks the localization boundary $r=e^{-j}$.
    The curves are schematic.}
    \label{fig:critical-y-split}
  \end{figure}
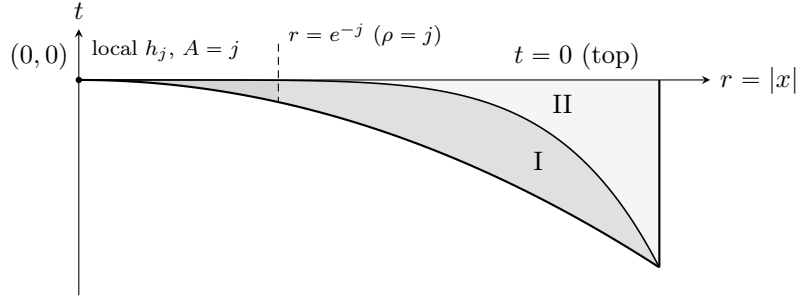

  \smallskip
  \noindent\emph{Step 1a: Case 1 (Region I):}
  Suppose that $y_0\le y<1$ and $\rho\ge1$.  If $0<\mu\le1$, then
  \eqref{eq:critical-profile-bounds} and
  \eqref{eq:critical-second-rho} give
  \begin{equation*}
    h_{\rho\rho}
    \le A\rho^{-\mu-2}
    \left(2\mu+\frac32\mu y-\frac12y^2\right)
    \le A\rho^{-\mu-2}
    \left(\frac72\mu-\frac12y_0^2\right).
  \end{equation*}
  The restriction arising from this case is
  $0<\mu\le y_0^2/14$.  We therefore set $\mu=\frac{y_0^2}{14}$.
  Since $y_0\in(0,1)$, this gives $0<\mu<1$, and
  \begin{equation*}
    h_{\rho\rho}
    \le-\frac14A\rho^{-\mu-2}y_0^2.
  \end{equation*}
  By \eqref{eq:critical-log-plap} and
  \eqref{eq:critical-time-lower}, $\Deltap h\le0<h_t$.  Thus, for every
  $A>0$,
  \begin{equation}\label{eq:critical-region-I}
    h_t-\Deltap h>0
    \qquad\text{in }\Theta_{l,1}\cap\{\rho\ge1,\ y_0\le y<1\}.
  \end{equation}

  \smallskip
  \noindent\emph{Step 1b: Case 2 (Region II): $y$ small.}
  It remains to consider $0<y<y_0$. Here the spatial term need not have the favorable sign, so we instead show that the positive time derivative dominates its possible positive part. Equations
  \eqref{eq:critical-first-rho}, \eqref{eq:critical-second-rho}, and
  \eqref{eq:critical-profile-bounds} give, for
  $\rho\ge1$ and $\mu\le1$,
  \begin{equation*}
    |h_\rho|
    \le A\rho^{-\mu-1}(\mu+y),
    \qquad
    h_{\rho\rho}\le\frac72A\mu\rho^{-\mu-2}.
  \end{equation*}
  Substituting these bounds into \eqref{eq:critical-log-plap} gives
  \begin{equation*}
    \Deltap h
    \le Cr^{-p}
    \bigl(A\rho^{-\mu-1}(\mu+y)\bigr)^{p-2}
    A\mu\rho^{-\mu-2}.
  \end{equation*}
  The definitions of $\rho$ and $y$ in
  \eqref{eq:critical-dimension-coordinates} give
  \begin{equation*}
    (-t)r^{-p}
    =\exp\left[-\rho\left(\frac ly-p\right)\right].
  \end{equation*}
  Combining the preceding estimate with \eqref{eq:critical-time-lower}, and
  using $\rho\ge1$ and $\mu+y\le2$, gives
  \begin{align}
    \frac{\Deltap h}{h_t}
     & \le CA^{p-2}\mu\rho^{1-p-\mu(p-2)}
    y^{-2}(-t)r^{-p} \notag               \\
     & \le CA^{p-2}\mu\rho^{1-p}y^{-2}
    \exp\left[-\rho\left(\frac ly-p\right)\right].
    \label{eq:critical-small-y-ratio}
  \end{align}
  Here and below $C$ may depend on $p$ and $l$, but not on $A$, $\rho$, or
  $y$.
  To obtain uniform decay for the right-hand side, we require
  \begin{equation*}
    y<  y_0=\frac{l}{2p}\in(0,1).
  \end{equation*}
  Since $0<y<y_0$, then $l/y-p\ge l/(2y)$.  The elementary inequality
  $s^2e^{-s}\le Ce^{-\nicefrac{s}{2}}$ (for large values exponential beats polynomial and for small values $C$ may be chosen), applied with $s=l\rho/(2y)$, gives
  \begin{align*}
    y^{-2}\exp\left[-\rho\left(\frac ly-p\right)\right]
     & \le y^{-2}e^{-\frac{l\rho}{2y}}     \\
     & \le C\rho^{-2}e^{-\frac{l\rho}{4y}}
    \le C\rho^{-2}e^{-\frac{p\rho}{2}}.
  \end{align*}
  Thus \eqref{eq:critical-small-y-ratio} yields
  \begin{equation}\label{eq:critical-region-II-ratio}
    \frac{\Deltap h}{h_t}
    \le CA^{p-2}\rho^{-p-1}e^{-\frac{p\rho}{2}}
    \qquad\text{in }\Theta_{l,1}\cap\{\rho\ge1,\ 0<y<y_0\}.
  \end{equation}
  The right-hand side of \eqref{eq:critical-region-II-ratio} tends to zero
  as $\rho\to\infty$ for each fixed $A>0$.
  Hence there is $R_A\ge1$ such that $h_t-\Deltap h>0$ in
  $\Theta_{l,1}\cap\{\rho\ge R_A,\ 0<y<y_0\}$.

  \smallskip
  \noindent\emph{Step 2: Amplification and pasting.}
  We now couple the amplitude $j$ with the localization $\rho>j$.
  For positive integers $j$, let
  \begin{equation*}
    h_j(x,t):=j\rho^{-\mu}F(y).
  \end{equation*}
  Applying \eqref{eq:critical-region-II-ratio} with $A=j$ on
  $\Theta_{l,1}\cap\{\rho\ge j,\ 0<y<y_0\}$ gives
  \begin{equation*}
    \frac{\Deltap h_j}{(h_j)_t}
    \le Cj^{p-2}\rho^{-p-1}e^{-\frac{p\rho}{2}}
    \le Cj^{-3}e^{-\frac{pj}{2}},
  \end{equation*}
  and $Cj^{-3}e^{-\frac{pj}{2}}\to0$ as $j\to\infty$.  Hence there exists an integer
  $j_0\ge1$, depending only on $p$ and $l$, such that, for every fixed
  $j\ge j_0$, $\Deltap h_j/(h_j)_t\le1$ throughout
  $\Theta_{l,1}\cap\{\rho\ge j,\ 0<y<y_0\}$.  Since $(h_j)_t>0$ by
  \eqref{eq:critical-time-lower}, this gives
  \begin{equation}\label{eq:critical-region-II}
    (h_j)_t-\Deltap h_j\ge0
    \qquad\text{in }\Theta_{l,1}\cap\{\rho\ge j,\ 0<y<y_0\}.
  \end{equation}

  Combining \eqref{eq:critical-region-I} with $A=j$ and
  \eqref{eq:critical-region-II}, for every $j\ge j_0$ we obtain
  \begin{equation*}
    (h_j)_t-\Deltap h_j\ge0
    \qquad\text{throughout }\Theta_{l,1}\cap\{\rho>j\}.
  \end{equation*}
  Since \eqref{eq:critical-first-rho} also shows that the spatial gradient
  does not vanish on $\Theta_{l,1}\cap\{\rho>j\}$, $h_j$ is a smooth
  classical supersolution, hence a
  continuous weak supersolution and therefore $p$-superparabolic, in
  $\Theta_{l,1}\cap\{\rho>j\}$.

  We now paste to a constant.  On $\Theta_{l,1}\cap\{\rho=j\}$,
  \begin{equation*}
    h_j(x,t)\ge m_j:=\frac12j^{1-\mu},
    \qquad
    m_j\to \infty.
  \end{equation*}
  The limit $m_j\to\infty$ follows from $m_j=\frac12j^{1-\mu}$ and
  $0<\mu<1$.  For every $j\ge j_0$, define
  \begin{equation*}
    w_j(x,t)=
    \begin{cases}
      \min\{h_j(x,t),m_j\}, & \rho>j,    \\
      m_j,                  & \rho\le j.
    \end{cases}
  \end{equation*}
  Since $h_j\ge m_j$ on $\Theta_{l,1}\cap\{\rho=j\}$, the limit of
  $\min\{h_j,m_j\}$ from $\rho>j$ equals $m_j$, matching the constant
  branch.  Thus $w_j$ is continuous across
  $\Theta_{l,1}\cap\{\rho=j\}$.  For every $j\ge j_0$, the pasting lemma applied to
  the local function $h_j$ and the constant $m_j$ shows that $w_j$ is
  positive and $p$-superparabolic in $\Theta_{l,1}$.  For
  $(x,t)\in\overline{\Theta_{l,1}}$ with $0<|x|\le e^{-j}$, we have
  $|t|\le |x|^l\le e^{-lj}$,
  so
  \begin{equation*}
    |(x,t)|\le\bigl(e^{-2j}+e^{-2lj}\bigr)^{\nicefrac{1}{2}}\to 0.
  \end{equation*}
  Given $k\ge1$, choose $j\ge j_0$ so that
  $(e^{-2j}+e^{-2lj})^{\nicefrac{1}{2}}<1/k$ and $m_j\ge k$.
  Every $\xi\in\partial\Theta_{l,1}$ with $|\xi|\ge1/k$
  then lies in $\{\rho<j\}$, where $w_j=m_j$ in a neighborhood relative
  to $\Theta_{l,1}$.
  Consequently,
  \begin{equation}\label{eq:critical-family-boundary}
    \liminf_{\Theta_{l,1}\ni\zeta\to\xi}w_j(\zeta)
    =m_j\ge k
    \qquad\text{for every }\xi\in\partial\Theta_{l,1}
    \text{ with }|\xi|\ge\frac1k.
  \end{equation}
  This verifies condition \eqref{eq:barrier-family} in
  \cref{def:barrier-family}.
  For each fixed $j$,
  sufficiently close to the origin we have $\rho>j$, and
  \begin{equation}\label{eq:critical-family-vanishing}
    0<w_j(x,t)\le h_j(x,t)
    \le j\rho^{-\mu}\to 0
  \end{equation}
  as $(x,t)\to(0,0)$ within $\Theta_{l,1}$.
  Equations \eqref{eq:critical-family-boundary} and
  \eqref{eq:critical-family-vanishing}, together with positivity and
  $p$-superparabolicity, verify \cref{def:barrier-family} for
  $\{w_j\}_{j\ge j_0}$ at the origin in $\Theta_{l,1}$.  Thus the origin is
  regular with respect to $\Theta_{l,1}$ by \cite[Theorem~3.3]{BBGP15}.
\end{proof}

\section{Regularity for $l=p$, $p_*\le p<2$}\label{sec:singular-endpoint}

By \cite[Lemma~3.6]{BB26}, we work in $\Theta_{p,1}$.  The two subsections below
share an amplification step.  Since $p<2$, a local function $h$ with
$h_t>0$ and $h_t-\Deltap h\ge0$ remains a supersolution after multiplication
by any $j\ge1$.  If $h>0$ and tends to zero at the origin, pasting these
multiples to increasing constants away from the origin yields a barrier
family.  The task in each subsection is to construct such an $h$ with the
supersolution inequality valid throughout a fixed neighborhood of the origin
in $\Theta_{p,1}$.

To analyze $h_t-\Deltap h\ge0$ in both subsections, set $r=|x|$ and
use the logarithmic radius and space-time ratio
\begin{equation}\label{eq:singular-coordinates}
  \rho=\log\frac1r,
  \qquad
  y=\frac{-t}{r^p}\in(0,1).
\end{equation}
The values $y=0$ and $y=1$ describe the top and initial boundary,
respectively.  The supersolution inequality must hold for every $0<y<1$.
We observe
\begin{equation*}
  \rho_r=-\frac1r,
  \qquad
  y_r=-\frac{py}{r},
  \qquad
  y_t=-r^{-p}.
\end{equation*}
For a smooth radial function, write $h=h(r,t)$. Since
$r^2h_{rr}=r(rh_r)_r-rh_r$, the radial $p$-Laplacian formula
\eqref{eq:radial-formula} can be written as
\begin{equation}\label{eq:endpoint-radial}
  \Deltap h=
  |h_r|^{p-2}
  \left((p-1)h_{rr}+\frac{n-1}{r}h_r\right)
  =r^{-p}|rh_r|^{p-2}
  \left[(p-1)r(rh_r)_r+(n-p)rh_r\right],
\end{equation}
whenever $h_r\ne0$.

\subsection{The strict singular range}\label{sec:singular-strict}

\begin{theorem}[Regularity in the strict singular range]\label{thm:singular-strict}
  Let $l=p$, $n\ge2$, $p_*<p<2$, and $\theta>0$.  Then the origin is
  regular with respect to $\Theta_{l,\theta}$.
\end{theorem}

The profile \eqref{eq:strict-h} uses two corrections in the ratio $y$
defined in \eqref{eq:singular-coordinates}: the higher power of $y$
makes the spatial term favorable away from $t=0$, while the lower
power supplies the time derivative needed near $t=0$.
\Cref{fig:singular-strict-geometry} shows the case split used when
$\Deltap h>0$.

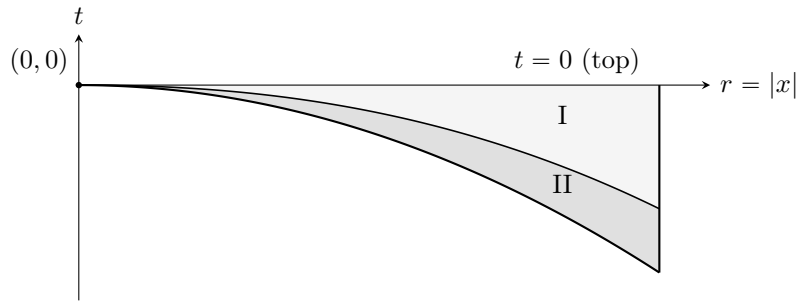
\begin{figure}[H]
  \centering
  \begin{tikzpicture}[x=1.6cm,y=1.6cm,>=stealth,
      every node/.style={font=\small}]
    \def\Rmax{4.8}
    \def\xmin{0.08}

    \path[fill=black!4]
    (0,0)--(\Rmax,0)
    -- plot[smooth,domain=\Rmax:\xmin,samples=90]
    (\x,{-1.55*(\x/\Rmax)^2/
        (2+ln(\Rmax/\x))^0.60})
    -- cycle;

    \path[fill=black!12]
    (0,0)
    -- plot[smooth,domain=\xmin:\Rmax,samples=90]
    (\x,{-1.55*(\x/\Rmax)^2/
        (2+ln(\Rmax/\x))^0.60})
    -- plot[smooth,domain=\Rmax:\xmin,samples=90]
    (\x,{-1.55*(\x/\Rmax)^2})
    -- cycle;

    \draw[thick]
    (0,0)
    -- plot[smooth,domain=\xmin:\Rmax,samples=90]
    (\x,{-1.55*(\x/\Rmax)^2});

    \draw[semithick]
    (0,0)
    -- plot[smooth,domain=\xmin:\Rmax,samples=90]
    (\x,{-1.55*(\x/\Rmax)^2/
        (2+ln(\Rmax/\x))^0.60});

    \draw[thick] (\Rmax,-1.55)--(\Rmax,0);

    \draw[->] (0,-1.78)--(0,0.42) node[above] {$t$};
    \draw[->] (0,0)--(5.22,0) node[right] {$r=|x|$};
    \fill (0,0) circle (1.2pt) node[above left] {$(0,0)$};

    \node[above] at (4.12,0) {$t=0$ (top)};

    \node[font=\small] at (4.0,-0.25) {I};

    \node[font=\small] at (4.0,-0.82) {II};

  \end{tikzpicture}

  \caption{The case split for $\Deltap h>0$ in \cref{thm:singular-strict}:
    Region I has $y\le\rho^{-\frac{1-d_1}{\beta_1}}$; Region II has
    $y>\rho^{-\frac{1-d_1}{\beta_1}}$, with $y$ and $\rho$ as in
    \eqref{eq:singular-coordinates}. The separating curve is schematic.}
  \label{fig:singular-strict-geometry}
\end{figure}

\begin{proof}
  By \cite[Lemma~3.6]{BB26}, it suffices to work in $\Theta_{p,1}$.
  \smallskip

  \noindent\emph{Step 1: Construct a local supersolution.}
  Let $a,d_1,d_2,\beta_1,\beta_2>0$, to be
  fixed later, and set
  \begin{equation}\label{eq:strict-h}
    h(r,t)=\rho^{-a}
    \left(1-\rho^{-d_1}y^{\beta_1}-\rho^{-d_2}y^{\beta_2}\right).
  \end{equation}

  \smallskip

  \noindent\emph{Derivative estimates.}
  Direct differentiation gives
  \begin{equation}\label{eq:strict-time}
    h_t=\frac{\rho^{-a}}{r^p}
    \left(\beta_1\rho^{-d_1}y^{\beta_1-1}
    +\beta_2\rho^{-d_2}y^{\beta_2-1}\right)>0
  \end{equation}
  and
  \begin{equation}\label{eq:strict-radial-first}
    rh_r=
    \rho^{-a}\left[\frac a\rho+
    \sum_{i=1}^2\rho^{-d_i}y^{\beta_i}
    \left(p\beta_i-\frac{a+d_i}{\rho}\right)\right].
  \end{equation}
  For all sufficiently large $\rho$, the coefficients
  $p\beta_i-(a+d_i)/\rho$ in \eqref{eq:strict-radial-first} are positive
  and bounded away from zero.  Hence
  \begin{equation}\label{eq:strict-radial-bound}
    rh_r\ge
    c\rho^{-a}
    \left(\rho^{-1}+\rho^{-d_1}y^{\beta_1}
    +\rho^{-d_2}y^{\beta_2}\right)>0.
  \end{equation}
  The constants may depend on $n,p$ and the chosen
  exponents, but not on $r,t,\rho$, or $y$.

  Substituting the two correction terms
  \begin{equation*}
    -\rho^{-a-d_i}y^{\beta_i},\qquad i=1,2,
  \end{equation*}
  into the radial $p$-Laplacian formula \eqref{eq:endpoint-radial}, we find
  that the leading coefficient of the $y^{\beta_i}$ contribution, i.e. the
  part of the coefficient containing no factor $\rho^{-1}$ or $\rho^{-2}$,
  is
  \begin{equation*}
    p\beta_i\bigl(n-p-p(p-1)\beta_i\bigr).
  \end{equation*}
  Thus the leading coefficient changes sign at
  \begin{equation*}
    \beta_i=\frac{n-p}{p(p-1)}.
  \end{equation*}
  We therefore choose
  \begin{equation}\label{eq:strict-power-signs}
    \beta_1<\frac{n-p}{p(p-1)}<\beta_2.
  \end{equation}
  By \eqref{eq:strict-power-signs} and the calculation in
  \cref{app:strict-calculation}, for sufficiently large $\rho$ and uniformly
  over $0<y<1$ in $\Theta_{p,1}$,
  \begin{equation}\label{eq:strict-spatial-sign}
    \Deltap h\le
    r^{-p}|rh_r|^{p-2}\rho^{-a}
    \left[C\left(\rho^{-1}+\rho^{-d_1}y^{\beta_1}\right)
    -c\rho^{-d_2}y^{\beta_2}\right].
  \end{equation}
  At points where $\Deltap h>0$, dropping the favorable negative $\beta_2$-term in
  \eqref{eq:strict-spatial-sign}
  and using the lower bound \eqref{eq:strict-radial-bound}, together with
  $p-2<0$, gives
  \begin{equation}\label{eq:strict-spatial-bound}
    0<\Deltap h\le
    Cr^{-p}\rho^{-a(p-1)}
    \left(\rho^{-1}+\rho^{-d_1}y^{\beta_1}\right)^{p-1}.
  \end{equation}
  Where $\Deltap h\le0$, the supersolution inequality follows from
  $h_t>0$ in \eqref{eq:strict-time}.  Where $\Deltap h>0$, we compare
  \eqref{eq:strict-spatial-bound} with \eqref{eq:strict-time} and split
  according to whether $\rho^{-d_1}y^{\beta_1}$ is at most or greater
  than $\rho^{-1}$, as shown in \cref{fig:singular-strict-geometry}.

  \smallskip

  \noindent\emph{Step 1a: Case 1, very close to the top (Region I in
    \cref{fig:singular-strict-geometry}).}
  At points of $\Theta_{p,1}$ with $\Deltap h>0$, assume
  \begin{equation*}
    \rho^{-d_1}y^{\beta_1}\le\rho^{-1},
    \qquad\text{equivalently}\qquad
    y\le\rho^{-\frac{1-d_1}{\beta_1}}.
  \end{equation*}
  We will choose $\beta_1<1$. Then
  \eqref{eq:strict-time} gives
  $h_t\ge\beta_1r^{-p}\rho^{-a-d_1}$, because
  $y^{\beta_1-1}\ge1$. Thus \eqref{eq:strict-spatial-bound} yields
  \begin{equation}\label{eq:strict-case1}
    \frac{\Deltap h}{h_t}
    \le C\rho^{a(2-p)+d_1-(p-1)}.
  \end{equation}
  The right-hand side of \eqref{eq:strict-case1} tends to zero as
  $\rho\to\infty$ provided
  \begin{equation}\label{eq:strict-requirement1}
    a(2-p)+d_1<p-1.
  \end{equation}
  Under \eqref{eq:strict-requirement1}, the ratio in
  \eqref{eq:strict-case1} is below one for sufficiently large $\rho$.
  Hence, for sufficiently large $\rho$ in $\Theta_{p,1}$,
  \begin{equation}\label{eq:strict-case1-conclusion}
    h_t-\Deltap h>0
    \quad\text{if }\Deltap h>0\text{ and }
    \rho^{-d_1}y^{\beta_1}\le\rho^{-1}.
  \end{equation}

  \smallskip

  \noindent\emph{Step 1b: Case 2, the remaining points with positive $p$-Laplacian
    (Region II in \cref{fig:singular-strict-geometry}).}
  At the remaining points of $\Theta_{p,1}$ with $\Deltap h>0$, suppose
  $\rho^{-d_1}y^{\beta_1}>\rho^{-1}$.
  The negative $\beta_2$-term in \eqref{eq:strict-spatial-sign}, together
  with $\Deltap h>0$, implies
  \begin{equation}\label{eq:strict-confinement}
    \rho^{-d_2}y^{\beta_2}
    \le C\left(\rho^{-1}+\rho^{-d_1}y^{\beta_1}\right).
  \end{equation}
  Since $\rho^{-d_1}y^{\beta_1}>\rho^{-1}$, the bound
  \eqref{eq:strict-confinement} simplifies to
  \begin{equation*}
    \rho^{-d_2}y^{\beta_2}\le C\rho^{-d_1}y^{\beta_1}.
  \end{equation*}
  If $d_2<d_1$, this implies that such points satisfy
  \begin{equation}\label{eq:strict-y-restriction}
    y\le C\rho^{-\frac{d_1-d_2}{\beta_2-\beta_1}}.
  \end{equation}
  Using \eqref{eq:strict-time} and \eqref{eq:strict-spatial-bound}, we obtain
  \begin{align}
    \frac{\Deltap h}{h_t}
     & \le C\rho^{(a+d_1)(2-p)}y^{1-\beta_1(2-p)}\notag \\
     & \le C\rho^{(a+d_1)(2-p)
      -\frac{d_1-d_2}{\beta_2-\beta_1}\left(1-\beta_1(2-p)\right)}.
    \label{eq:strict-case2}
  \end{align}
  The exponent of $y$ is positive because $\beta_1<1$ and $0<2-p<1$.
  The right-hand side of \eqref{eq:strict-case2} tends to zero as
  $\rho\to\infty$ provided
  \begin{equation}\label{eq:strict-requirement2}
    (a+d_1)(2-p)
    <\frac{d_1-d_2}{\beta_2-\beta_1}
    \left(1-\beta_1(2-p)\right).
  \end{equation}
  Under \eqref{eq:strict-requirement2}, the ratio in
  \eqref{eq:strict-case2} is below one for sufficiently large $\rho$.
  Hence, for sufficiently large $\rho$ in $\Theta_{p,1}$,
  \begin{equation}\label{eq:strict-case2-conclusion}
    h_t-\Deltap h>0
    \quad\text{if }\Deltap h>0\text{ and }
    \rho^{-d_1}y^{\beta_1}>\rho^{-1}.
  \end{equation}

  \smallskip

  \noindent\emph{Step 2: Parameter choices.}
  The assumption $p>p_*$ is exactly what gives
  \begin{equation*}
    \frac{n-p}{p(p-1)}<\frac1{2-p}.
  \end{equation*}

  The interval for $\beta_2$ is therefore nonempty.  Choose
  \begin{equation*}
    0<\beta_1<\min\left\{1,\frac{n-p}{p(p-1)}\right\},
    \qquad
    \frac{n-p}{p(p-1)}<\beta_2<\frac1{2-p},
  \end{equation*}
  and fix any $0<d_1<p-1$.
  To verify that suitable choices can be made, first set $a=d_2=0$.  At
  these values, the margin in \eqref{eq:strict-requirement1} is
  $p-1-d_1>0$, and the right-hand side minus the left-hand side of
  \eqref{eq:strict-requirement2} is
  \begin{equation*}
    \frac{d_1\left(1-\beta_2(2-p)\right)}{\beta_2-\beta_1}>0.
  \end{equation*}
  By continuity, both inequalities persist for sufficiently small positive
  $a,d_2$, chosen with $d_2<d_1$.  This fixes all five exponents.  The zero
  values were used only to exhibit the positive margins.

  \begin{samepage}
    After fixing the five exponents, choose $\rho_0\ge1$ large enough for
    the following properties.  Whenever $\rho\ge\rho_0$, the bracket in
    \eqref{eq:strict-h} lies between $1/2$ and $1$ for $0\le y\le1$.
    For $\rho\ge\rho_0$ and $0<y<1$, the estimates
    \eqref{eq:strict-radial-bound} and \eqref{eq:strict-spatial-sign}
    hold.  The right-hand sides of \eqref{eq:strict-case1} and
    \eqref{eq:strict-case2} are below one in their respective cases
    whenever $\rho\ge\rho_0$.
    The choice of $\rho_0$ is uniform in $y$.  The profile therefore
    satisfies
    \begin{equation}\label{eq:strict-profile-bounds}
      \tfrac12\rho^{-a}\le h\le\rho^{-a}
      \qquad(\rho\ge\rho_0,\ 0\le y\le1).
    \end{equation}
    If $\Deltap h\le0$, then \eqref{eq:strict-time} gives
    $h_t-\Deltap h>0$.  At points with $\Deltap h>0$, the estimate
    \eqref{eq:strict-spatial-bound} leads to the two exhaustive
    conclusions \eqref{eq:strict-case1-conclusion} and
    \eqref{eq:strict-case2-conclusion}.  The choice of $\rho_0$ makes
    both conclusions valid throughout $\Theta_{p,1}\cap\{\rho>\rho_0\}$.
    Equations \eqref{eq:strict-time} and
    \eqref{eq:strict-radial-bound} also give $h_t>0$ and $h_r>0$.
    Therefore
    \begin{equation}\label{eq:strict-local-supersolution}
      h_t-\Deltap h\ge0,\qquad h_t>0,\qquad h_r>0
      \quad\text{in }\Theta_{p,1}\cap\{\rho>\rho_0\}.
    \end{equation}
    Thus $h$ is a positive smooth classical supersolution in
    $\Theta_{p,1}\cap\{\rho>\rho_0\}$.
  \end{samepage}

  \smallskip

  \noindent\emph{Step 3: Amplify and paste.}
  The cutoff $\rho_0$ in \eqref{eq:strict-local-supersolution} is fixed
  before choosing the amplitude $j$.  Since $1<p<2$, for every $j\ge1$
  \begin{equation}\label{eq:singular-amplification}
    (jh)_t-\Deltap(jh)
    =j^{p-1}\bigl[(j^{2-p}-1)h_t+(h_t-\Deltap h)\bigr]\ge0.
  \end{equation}
  Since $h_r>0$ by \eqref{eq:strict-local-supersolution}, $jh$ is a
  smooth classical supersolution, hence $p$-superparabolic, in
  $\Theta_{p,1}\cap\{\rho>\rho_0\}$.

  The formula \eqref{eq:strict-h} extends continuously to $y=0$ and
  $y=1$ for $r>0$, and \eqref{eq:strict-profile-bounds} makes this
  extension positive on the boundary of
  $\Theta_{p,1}\cap\{\rho>\rho_0\}$ away from the origin.  In particular,
  with $c_0:=\frac12\rho_0^{-a}>0$,
  \begin{equation}\label{eq:strict-interface-lower}
    h(|x|,t)\ge c_0
    \qquad\text{on }\overline{\Theta}_{p,1}
    \cap\{|x|=e^{-\rho_0}\}.
  \end{equation}
  For every positive integer $j$, define
  \begin{equation*}
    w_j(x,t)=
    \begin{cases}
      \min\{jh(|x|,t),jc_0\}, & |x|<e^{-\rho_0},    \\
      jc_0,                   & |x|\ge e^{-\rho_0}.
    \end{cases}
  \end{equation*}
  By \eqref{eq:strict-interface-lower}, the limit of the inner branch
  $\min\{jh,jc_0\}$ at $|x|=e^{-\rho_0}$ equals $jc_0$, matching the outer
  branch.  The pasting lemma \cite[Lemma~2.9]{BBGP15} therefore makes $w_j$
  positive and $p$-superparabolic in $\Theta_{p,1}$.

  \smallskip

  \noindent\emph{Step 4: Verify the barrier-family conditions.}
  For each fixed $j$, \eqref{eq:strict-profile-bounds} gives
  \begin{equation}\label{eq:strict-family-vanishing}
    0<w_j(x,t)\le jh(|x|,t)\le j\rho^{-a}\to0
    \qquad\text{as }(x,t)\to(0,0)\text{ within }\Theta_{p,1}.
  \end{equation}
  For each integer $k\ge1$, the set
  \begin{equation*}
    K_k:=\{\xi\in\partial\Theta_{p,1}:|\xi|\ge1/k\}
  \end{equation*}
  is compact.  The formula for $w_1$ and
  \eqref{eq:strict-profile-bounds} give it a continuous positive trace
  on $K_k$.  Write $w_1(\xi)$ for this trace, set
  $c_k:=\min_{\xi\in K_k}w_1(\xi)>0$, and choose an integer
  $j\ge k/c_k$.  Since $w_j=jw_1$,
  \begin{equation}\label{eq:strict-family-boundary}
    \liminf_{\Theta_{p,1}\ni\zeta\to\xi}w_j(\zeta)
    =j w_1(\xi)\ge j c_k\ge k
    \qquad\text{for every }\xi\in K_k.
  \end{equation}
  Equations \eqref{eq:strict-family-vanishing} and
  \eqref{eq:strict-family-boundary}, together with positivity and
  $p$-superparabolicity, verify \cref{def:barrier-family} in
  $\Theta_{p,1}$ at the origin.  The barrier-family characterization
  \cite[Theorem~3.3]{BBGP15} proves the theorem.
\end{proof}

\subsection{The double-critical endpoint}\label{sec:double-critical}

\begin{theorem}[Regularity at the double-critical endpoint]\label{thm:double-critical}
  Let $l=p$, $n\ge2$, $p=p_*$, and $\theta>0$.  Then the origin
  is regular with respect to $\Theta_{l,\theta}$.
\end{theorem}

At $p=p_*$, the bounds
$\frac{n-p}{p(p-1)}<\beta_2<\frac{1}{2-p}$ used in Step 2 of the proof of
\cref{thm:singular-strict} meet at $\frac{n+1}{2}$.  Thus the pure-power
choice of $\beta_2$ in that proof is unavailable.  With $y$ and $\rho$
defined in \eqref{eq:singular-coordinates}, we instead use the two
corrections in \eqref{eq:critical-h},
\begin{equation*}
  \rho^{-\nicefrac{1}{2}}y^{\nicefrac{1}{2}}
  \quad\text{and}\quad
  y^{\frac{n+1}{2}}
  \left(1+\frac{1+\log(1/y)}{\sqrt{\log\rho}}\right)^2.
\end{equation*}
The exponent $(n+1)/2$ makes the leading contribution of
$y^{\frac{n+1}{2}}$ to $\Deltap h$ vanish.  The logarithmic factor in
the second correction makes the remaining spatial contribution small
relative to $h_t$.  For $y\le\rho^{-1}$, the
$\rho^{-\nicefrac{1}{2}}y^{\nicefrac{1}{2}}$ correction supplies a lower bound for $h_t$.
For $\rho^{-1}<y\le\rho^{-\frac{1}{2n}}$, the squared logarithmic factor is
at least a constant times $\log\rho$.  For $y>\rho^{-\frac{1}{2n}}$, the
correction containing $y^{\frac{n+1}{2}}$ and the squared logarithmic
factor exceeds the $\rho^{-\nicefrac{1}{2}}y^{\nicefrac{1}{2}}$ correction by
at least a factor $\rho^{\nicefrac{1}{4}}$.  These three regions are shown in
\cref{fig:double-critical-geometry}.

\begin{figure}[H]
  \centering
  \begin{tikzpicture}[x=1.6cm,y=1.6cm,>=stealth,
      every node/.style={font=\small}]
    \def\Rmax{4.8}
    \def\xmin{0.08}

    \path[fill=black!4]
    (0,0)--(\Rmax,0)
    -- plot[smooth,domain=\Rmax:\xmin,samples=90]
    (\x,{-1.55*(\x/\Rmax)^2/
        (2+ln(\Rmax/\x))})
    -- cycle;

    \path[fill=black!10]
    (0,0)
    -- plot[smooth,domain=\xmin:\Rmax,samples=90]
    (\x,{-1.55*(\x/\Rmax)^2/
        (2+ln(\Rmax/\x))})
    -- plot[smooth,domain=\Rmax:\xmin,samples=90]
    (\x,{-1.55*(\x/\Rmax)^2/
        (2+ln(\Rmax/\x))^0.25})
    -- cycle;

    \path[fill=black!16]
    (0,0)
    -- plot[smooth,domain=\xmin:\Rmax,samples=90]
    (\x,{-1.55*(\x/\Rmax)^2/
        (2+ln(\Rmax/\x))^0.25})
    -- plot[smooth,domain=\Rmax:\xmin,samples=90]
    (\x,{-1.55*(\x/\Rmax)^2})
    -- cycle;

    \draw[thick]
    (0,0)
    -- plot[smooth,domain=\xmin:\Rmax,samples=90]
    (\x,{-1.55*(\x/\Rmax)^2});

    \draw[semithick]
    (0,0)
    -- plot[smooth,domain=\xmin:\Rmax,samples=90]
    (\x,{-1.55*(\x/\Rmax)^2/
        (2+ln(\Rmax/\x))^0.25});

    \draw[semithick]
    (0,0)
    -- plot[smooth,domain=\xmin:\Rmax,samples=90]
    (\x,{-1.55*(\x/\Rmax)^2/
        (2+ln(\Rmax/\x))});

    \draw[thick] (\Rmax,-1.55)--(\Rmax,0);

    \draw[->] (0,-1.78)--(0,0.42) node[above] {$t$};
    \draw[->] (0,0)--(5.22,0) node[right] {$r=|x|$};
    \fill (0,0) circle (1.2pt) node[above left] {$(0,0)$};

    \node[above] at (4.12,0) {$t=0$ (top)};

    \node[font=\small] at (4.05,-0.20) {I};

    \node[font=\small] at (4.05,-0.68) {II};

    \node[font=\scriptsize] at (4.55,-1.27) {III};

  \end{tikzpicture}

  \caption{The three regions in \cref{thm:double-critical}, separated by
  $t=-r^p/\log(1/r)$ and $t=-r^p/[\log(1/r)]^{\frac{1}{2n}}$.
  The curves are schematic.}
  \label{fig:double-critical-geometry}
\end{figure}
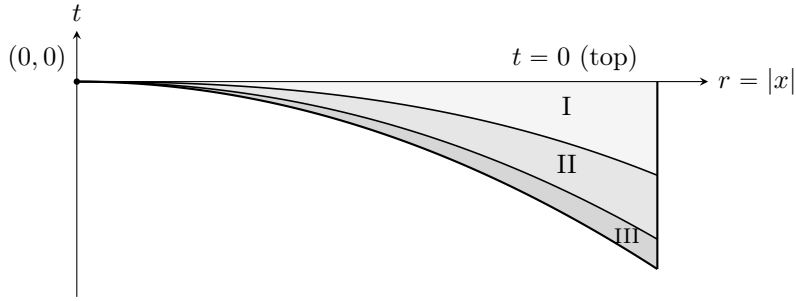

\begin{proof}
  \noindent\emph{Step 1: Construct $h$ and estimate its derivatives.}
  Let $\varepsilon>0$, to be fixed below. For $\rho>1$, set
  \begin{equation}\label{eq:critical-h}
    h(r,t)=(\log\rho)^{-\nicefrac{1}{4}}
    \left[1-\varepsilon\rho^{-\nicefrac{1}{2}}y^{\nicefrac{1}{2}}
    -\varepsilon y^{\frac{n+1}{2}}
    \left(1+\frac{1+\log(1/y)}{\sqrt{\log\rho}}\right)^2\right].
  \end{equation}
  For $\log\rho\ge1$ and $0<y<1$, we have
  $\rho^{-\nicefrac{1}{2}}y^{\nicefrac{1}{2}}\le1$.  The log correction
  in \eqref{eq:critical-h} satisfies
  \begin{equation*}
    y^{\frac{n+1}{2}}
    \left(1+\frac{1+\log(1/y)}{\sqrt{\log\rho}}\right)^2
    \le y^{\frac{n+1}{2}}(2+\log(1/y))^2.
  \end{equation*}
  The right-hand side is bounded for $0<y<1$, since it tends to zero
  as $y\to0^+$.  Fix $\varepsilon>0$, depending only on $n$, so small
  that the bracket in
  \eqref{eq:critical-h} lies between $1/2$ and $1$.

  For all sufficiently large $\rho$, uniformly in $0<y<1$, we have
  \begin{equation}\label{eq:critical-time}
    h_t\ge\frac{c}{r^p(\log\rho)^{\nicefrac{1}{4}}}
    \left[\rho^{-\nicefrac{1}{2}}y^{-\nicefrac{1}{2}}
    +y^{\frac{n-1}{2}}
    \left(1+\frac{1+\log(1/y)}{\sqrt{\log\rho}}\right)^2\right]>0,
  \end{equation}
  and
  \begin{equation}\label{eq:critical-radial-bound}
    rh_r\ge\frac{c}{(\log\rho)^{\nicefrac{1}{4}}}
    \left[\frac1{\rho\log\rho}+\rho^{-\nicefrac{1}{2}}y^{\nicefrac{1}{2}}
    +y^{\frac{n+1}{2}}
    \left(1+\frac{1+\log(1/y)}{\sqrt{\log\rho}}\right)^2\right]>0.
  \end{equation}
  The $p$-Laplacian of $h$ satisfies
  \begin{equation}\label{eq:critical-spatial}
    |\Deltap h|\le
    \frac{Cr^{-p}|rh_r|^{p-2}}{(\log\rho)^{\nicefrac{1}{4}}}
    \left[\frac1{\rho\log\rho}+\rho^{-\nicefrac{1}{2}}y^{\nicefrac{1}{2}}
    +\frac{y^{\frac{n+1}{2}}\bigl(\sqrt{\log\rho}+1+\log(1/y)\bigr)}{\log\rho}
    \right].
  \end{equation}
  The constants $c,C>0$ in \eqref{eq:critical-time}--\eqref{eq:critical-spatial} depend only on $n$ and the fixed
  $\varepsilon$.  These three estimates hold uniformly for $0<y<1$
  in $\Theta_{p,1}$ when $\rho$ is sufficiently large, as proved in
  \cref{app:critical-derivatives}.  Since $p=p_*$,
  \begin{equation*}
    n-p=p(p-1)\frac{n+1}{2},
  \end{equation*}
  so the leading contribution of $y^{\frac{n+1}{2}}$ to $\Deltap h$
  vanishes.  The first and third terms in \eqref{eq:critical-spatial} bound the
  remaining contribution of the logarithmic correction, as shown in
  \cref{app:critical-derivatives}.

  \smallskip

  \noindent\emph{Step 2: Compare $|\Deltap h|$ with $h_t$.}
  Dividing \eqref{eq:critical-spatial} by $h_t>0$ displays the three
  contributions to be estimated:
  \begin{align}
    \label{eq:main-log}
    \frac{|\Deltap h|}{h_t}
    \le
    \frac{Cr^{-p}|rh_r|^{p-2}}{h_t(\log\rho)^{\nicefrac{1}{4}}}
    \left[
    \frac1{\rho\log\rho}
    +\rho^{-\nicefrac{1}{2}}y^{\nicefrac{1}{2}}
    +y^{\frac{n+1}{2}}
    \frac{\sqrt{\log\rho}+1+\log(1/y)}{\log\rho}
    \right].
  \end{align}
  We use the lower bounds for $h_t$ and $rh_r$ in
  \eqref{eq:critical-time} and \eqref{eq:critical-radial-bound}.
  Because $p-2<0$, the lower bound for $rh_r$ gives an upper bound
  for $|rh_r|^{p-2}$.

  \smallskip
  \noindent\emph{Step 2a: The $y^{\frac{n+1}{2}}$ contribution.}
  Using the $y^{\frac{n-1}{2}}$ term in \eqref{eq:critical-time} and the
  $y^{\frac{n+1}{2}}$ term in \eqref{eq:critical-radial-bound}, the
  appendix estimate \eqref{eq:critical-contribution-appendix} gives
  \begin{equation}\label{eq:critical-contribution}
    \frac{Cr^{-p}|rh_r|^{p-2}}{h_t(\log\rho)^{\nicefrac{1}{4}}}
    y^{\frac{n+1}{2}}
    \frac{\sqrt{\log\rho}+1+\log(1/y)}{\log\rho}
    \le C(\log\rho)^{-\nicefrac{p}{4}}.
  \end{equation}
  Thus the third contribution in \eqref{eq:main-log} tends to zero
  uniformly for $0<y<1$.

  For $y>\rho^{-1}$ and sufficiently large $\rho$, the first term in
  the brackets of \eqref{eq:main-log} is bounded by the second:
  \begin{equation}\label{eq:critical-radial-dominated}
    \frac1{\rho\log\rho}\le\rho^{-1}
    <\rho^{-\nicefrac{1}{2}}y^{\nicefrac{1}{2}}.
  \end{equation}
  Both terms have the same positive prefactor in \eqref{eq:main-log},
  so the first contribution is no larger than the second whenever
  $y>\rho^{-1}$.

  We estimate the contributions of $1/(\rho\log\rho)$ and
  $\rho^{-\nicefrac{1}{2}}y^{\nicefrac{1}{2}}$ in the three regions of
  \cref{fig:double-critical-geometry}.

  \noindent\emph{Step 2b: Case 1, very near the top (Region I in
    \cref{fig:double-critical-geometry}).}
  Suppose that $0<y\le\rho^{-1}$.  We first estimate the contribution
  of $\rho^{-\nicefrac{1}{2}}y^{\nicefrac{1}{2}}$ in \eqref{eq:main-log}.  Since $\rho>1$,
  we have $y\le\rho^{-1}\le\rho^{-\frac{1}{2n}}$, so
  \eqref{eq:critical-lower-small-denominator} applies.  Moreover,
  $\log(1/y)\ge\log\rho$, and hence
  \begin{equation*}
    \left(1+\frac{1+\log(1/y)}{\sqrt{\log\rho}}\right)^2
    \ge\log\rho.
  \end{equation*}
  Combining the lower bound for the squared logarithmic factor with
  \eqref{eq:lower-contribution} and
  \eqref{eq:critical-lower-small-denominator} gives, for the full contribution
  of $\rho^{-\nicefrac{1}{2}}y^{\nicefrac{1}{2}}$ in \eqref{eq:main-log},
  \begin{equation}\label{eq:critical-lower-small-y}
    \frac{Cr^{-p}|rh_r|^{p-2}}{h_t(\log\rho)^{\nicefrac{1}{4}}}
    \rho^{-\nicefrac{1}{2}}y^{\nicefrac{1}{2}}
    \le
    C(\log\rho)^{\frac{2-p}{4}}
    \left(1+\frac{1+\log(1/y)}{\sqrt{\log\rho}}\right)^{-2(2-p)}
    \le C(\log\rho)^{-\frac{3(2-p)}4}.
  \end{equation}
  We next estimate the contribution of $1/(\rho\log\rho)$ in
  \eqref{eq:main-log}.  The condition $y\le\rho^{-1}$ also gives
  \begin{equation*}
    \rho^{\nicefrac{1}{2}}y^{\nicefrac{1}{2}}\le1.
  \end{equation*}
  Combining $\rho^{\nicefrac{1}{2}}y^{\nicefrac{1}{2}}\le1$ with
  \eqref{eq:critical-radial-appendix} gives
  \begin{equation}\label{eq:critical-radial-contribution}
    \begin{aligned}
      \frac{Cr^{-p}|rh_r|^{p-2}}{h_t(\log\rho)^{\nicefrac{1}{4}}}
      \frac1{\rho\log\rho}
       & \le C(\log\rho)^{\frac{2-p}{4}}(\rho\log\rho)^{-(p-1)}
      \rho^{\nicefrac{1}{2}}y^{\nicefrac{1}{2}}
      \\
       & \le C(\log\rho)^{\frac{2-p}{4}}(\rho\log\rho)^{-(p-1)}.
    \end{aligned}
  \end{equation}
  The bounds \eqref{eq:critical-lower-small-y} and
  \eqref{eq:critical-radial-contribution} tend to zero uniformly for
  $0<y\le\rho^{-1}$.

  \smallskip
  \noindent\emph{Step 2c: Case 2, intermediate region (Region II in
    \cref{fig:double-critical-geometry}).}
  Suppose that $\rho^{-1}<y\le\rho^{-\frac{1}{2n}}$.  We first estimate the
  contribution of $\rho^{-\nicefrac{1}{2}}y^{\nicefrac{1}{2}}$ in \eqref{eq:main-log}.
  The condition $y\le\rho^{-\frac{1}{2n}}$ gives
  $1+\log(1/y)\ge(\log\rho)/(2n)$, and hence
  \begin{equation}\label{eq:critical-intermediate-log-gain}
    \left(1+\frac{1+\log(1/y)}{\sqrt{\log\rho}}\right)^2
    \ge\frac{\log\rho}{4n^2}.
  \end{equation}
  The first inequality in \eqref{eq:critical-lower-small-y} follows
  from \eqref{eq:lower-contribution} and
  \eqref{eq:critical-lower-small-denominator} and remains valid
  for $\rho^{-1}<y\le\rho^{-\frac{1}{2n}}$.  Together with
  \eqref{eq:critical-intermediate-log-gain}, it bounds the full
  contribution of $\rho^{-\nicefrac{1}{2}}y^{\nicefrac{1}{2}}$ in \eqref{eq:main-log} by
  $C(\log\rho)^{-\frac{3(2-p)}4}$, uniformly for
  $\rho^{-1}<y\le\rho^{-\frac{1}{2n}}$.

  By \eqref{eq:critical-radial-dominated}, the contribution of
  $1/(\rho\log\rho)$ is no larger than the preceding one.  Thus both
  contributions tend to zero uniformly for
  $\rho^{-1}<y\le\rho^{-\frac{1}{2n}}$.

  \smallskip
  \noindent\emph{Step 2d: Case 3, farther below the top (Region III in
    \cref{fig:double-critical-geometry}).}
  Suppose that $\rho^{-\frac{1}{2n}}<y<1$.  To bound the contribution of
  $\rho^{-\nicefrac{1}{2}}y^{\nicefrac{1}{2}}$ in \eqref{eq:main-log}, we first compare the
  two corrections in \eqref{eq:critical-h}:
  \begin{equation}\label{eq:critical-power-dominance}
    \frac{y^{\frac{n+1}{2}}
    \left(1+\dfrac{1+\log(1/y)}{\sqrt{\log\rho}}\right)^2}
    {\rho^{-\nicefrac{1}{2}}y^{\nicefrac{1}{2}}}
    \ge\rho^{\nicefrac{1}{2}}y^{\nicefrac{n}{2}}>\rho^{\nicefrac{1}{4}}.
  \end{equation}
  The first inequality uses that the squared logarithmic factor in
  \eqref{eq:critical-h} is at least one, while the second uses
  $y>\rho^{-\frac{1}{2n}}$.
  Combining \eqref{eq:lower-contribution} and
  \eqref{eq:critical-lower-large-denominator} bounds the contribution
  of $\rho^{-\nicefrac{1}{2}}y^{\nicefrac{1}{2}}$ in \eqref{eq:main-log}.  Then
  \eqref{eq:critical-power-dominance} gives
  \begin{equation}\label{eq:critical-lower-large-y}
    \frac{Cr^{-p}|rh_r|^{p-2}}{h_t(\log\rho)^{\nicefrac{1}{4}}}
    \rho^{-\nicefrac{1}{2}}y^{\nicefrac{1}{2}}
    \le
    C(\log\rho)^{\frac{2-p}{4}}\rho^{-\nicefrac{1}{2}}y^{-\nicefrac{n}{2}}
    \le C(\log\rho)^{\frac{2-p}{4}}\rho^{-\nicefrac{1}{4}}.
  \end{equation}
  The last bound in \eqref{eq:critical-lower-large-y} tends to zero
  uniformly for $\rho^{-\frac{1}{2n}}<y<1$.
  Since $\rho^{-\frac{1}{2n}}>\rho^{-1}$ for $\rho>1$, the Case~3
  assumption implies $y>\rho^{-1}$.  By
  \eqref{eq:critical-radial-dominated}, the contribution of
  $1/(\rho\log\rho)$ is no larger than the one bounded in
  \eqref{eq:critical-lower-large-y}.  Thus both contributions tend to
  zero uniformly for $\rho^{-\frac{1}{2n}}<y<1$.

  \smallskip

  \noindent\emph{Step 3: Fix a local supersolution.}
  The third contribution in \eqref{eq:main-log} tends to zero uniformly
  by \eqref{eq:critical-contribution}.  The second does so in Region I
  by \eqref{eq:critical-lower-small-y}, in Region II by
  \eqref{eq:critical-lower-small-denominator} and
  \eqref{eq:critical-intermediate-log-gain}, and in Region III by
  \eqref{eq:critical-lower-large-y}.  The first contribution tends to
  zero in Region I by \eqref{eq:critical-radial-contribution} and is
  bounded by the second in Regions II and III by
  \eqref{eq:critical-radial-dominated}.
  Hence \eqref{eq:main-log} gives
  \begin{equation}\label{eq:critical-ratio-limit}
    \frac{|\Deltap h|}{h_t}\to0
    \qquad\text{as }\rho\to\infty\text{ uniformly for }0<y<1.
  \end{equation}
  Choose $\rho_0$ so large that \eqref{eq:critical-time},
  \eqref{eq:critical-radial-bound}, and
  \eqref{eq:critical-ratio-limit} give
  \begin{equation}\label{eq:critical-local-supersolution}
    h_t-\Deltap h\ge0,\qquad h_t>0,\qquad h_r>0
    \quad\text{in }\Theta_{p,1}\cap\{\rho>\rho_0\}.
  \end{equation}
  The choice of $\varepsilon$ after \eqref{eq:critical-h} also gives
  \begin{equation}\label{eq:double-critical-profile-bounds}
    \tfrac12(\log\rho)^{-\nicefrac{1}{4}}\le h\le(\log\rho)^{-\nicefrac{1}{4}}
    \qquad(\rho\ge\rho_0,\ 0\le y\le1).
  \end{equation}

  \smallskip

  \noindent\emph{Step 4: Amplify and paste.}
  Set $c_0:=\tfrac12(\log\rho_0)^{-\nicefrac{1}{4}}$.  By
  \eqref{eq:double-critical-profile-bounds}, $h\ge c_0$ on
  $\overline{\Theta}_{p,1}\cap\{\rho=\rho_0\}$.  For each positive
  integer $j$, define
  \begin{equation*}
    w_j(x,t)=
    \begin{cases}
      \min\{jh(|x|,t),jc_0\}, & \rho>\rho_0,   \\
      jc_0,                   & \rho\le\rho_0.
    \end{cases}
  \end{equation*}
  On $\overline{\Theta}_{p,1}\cap\{\rho=\rho_0\}$, the inner branch
  $\min\{jh,jc_0\}$ equals the outer value $jc_0$.  By
  \eqref{eq:critical-local-supersolution}, the amplification identity
  \eqref{eq:singular-amplification}, and the pasting lemma
  \cite[Lemma~2.9]{BBGP15}, each $w_j$ is positive and
  $p$-superparabolic in $\Theta_{p,1}$.

  \smallskip

  \noindent\emph{Step 5: Verify the barrier-family conditions.}
  Near the origin,
  \eqref{eq:double-critical-profile-bounds} gives
  $0<w_j\le j(\log\rho)^{-\nicefrac{1}{4}}\to0$ as $(x,t)\to(0,0)$ within
  $\Theta_{p,1}$.
  For each integer $k\ge1$, set
  $K_k:=\{\xi\in\partial\Theta_{p,1}:|\xi|\ge1/k\}$.
  The formula for $w_1$ and \eqref{eq:double-critical-profile-bounds}
  give it a continuous positive trace on the compact set $K_k$.
  Set $c_k:=\min_{\xi\in K_k}w_1(\xi)>0$ and choose an integer
  $j\ge k/c_k$.  Since
  $w_j=jw_1$, we have
  \begin{equation*}
    \liminf_{\Theta_{p,1}\ni\zeta\to\xi}w_j(\zeta)
    =j w_1(\xi)\ge k
    \qquad\text{for every }\xi\in K_k.
  \end{equation*}
  Thus $(w_j)_{j\ge1}$ satisfies \cref{def:barrier-family} in
  $\Theta_{p,1}$ at the origin.  The barrier-family characterization
  \cite[Theorem~3.3]{BBGP15} proves the theorem.
\end{proof}

The preceding results and the known cases give the classification.
\begin{proof}[Proof of \cref{thm:classification}]
  If $1<p<p_*$, \cite[Proposition~1.4]{BB26} shows that the origin
  is irregular for $l\le p$, while \cite[Proposition~1.2]{BB26} gives
  regularity for $l>p$.  Suppose next that
  $p_*\le p<2$.  By \cref{thm:irregularity-below}, the origin is irregular
  for $l<p$, while \cite[Proposition~1.2]{BB26} gives regularity for $l>p$.
  At the remaining endpoint $l=p$, \cref{thm:double-critical} applies when
  $p=p_*$, and \cref{thm:singular-strict} applies when $p_*<p<2$.
  The cases with $p=2$ are given by \cite[Proposition~1.5]{BB26}.

  For $2<p<n$, \cref{thm:irregularity-below} gives irregularity for
  $l<p$, \cref{thm:degenerate-endpoint} gives regularity at $l=p$, and
  regularity for $l>p$ follows from \cite[Proposition~4.2]{BBGP15}.
  For $0<l<p$ with $2<p=n$, \cref{thm:critical-dimension} gives regularity;
  \cref{thm:degenerate-endpoint} treats $l=p$, and
  \cite[Proposition~4.2]{BBGP15} treats $l>p$.  Finally,
  \cite[Proposition~1.3]{BB26} gives regularity for every $l>0$ when
  $p>n$.  This proves \cref{thm:classification}.
\end{proof}

\appendix

\section{Technical estimates for Section~5}
\label{app:section5-estimates}

Throughout this appendix, $r$, $\rho$, and $y$ are defined in
\eqref{eq:singular-coordinates}. Write $D:=r\partial_r$, so that
$D^2f=r(rf_r)_r$.

\subsection{The strict singular calculation}
\label{app:strict-calculation}

We establish \eqref{eq:strict-spatial-sign}, which states
that, for sufficiently large $\rho$ and uniformly for $0<y<1$,
\begin{equation*}
  \Deltap h\le
  r^{-p}|rh_r|^{p-2}\rho^{-a}
  \left[C\left(\rho^{-1}+\rho^{-d_1}y^{\beta_1}\right)
  -c\rho^{-d_2}y^{\beta_2}\right].
\end{equation*}
By  \eqref{eq:endpoint-radial},
\begin{equation*}
 \Deltap h=r^{-p}|rh_r|^{p-2}\mathcal R[h],\quad  \mathcal R[f]:=(p-1)D^2f+(n-p)Df.
\end{equation*}
Since $rh_r>0$ by \eqref{eq:strict-radial-bound}, the factor
$r^{-p}|rh_r|^{p-2}$ is positive. Thus
\eqref{eq:strict-spatial-sign} follows from
\eqref{eq:strict-spatial-bracket} below. Write \eqref{eq:strict-h} as
\begin{equation}\tag{\ref{eq:strict-h}$'$}\label{eq:strict-h-expanded}
  h=\rho^{-a}-\rho^{-a-d_1}y^{\beta_1}
  -\rho^{-a-d_2}y^{\beta_2}.
\end{equation}
We apply $\mathcal R$ to $\rho^{-a}$ and the two subtracted terms in
\eqref{eq:strict-h-expanded} by calculating
$\mathcal R[\rho^{-\lambda}y^\beta]$.

By \eqref{eq:singular-coordinates}, $D\rho=-1$ and $Dy=-py$. Hence
$D(\rho^{-\lambda}y^\beta)
=\rho^{-\lambda}y^\beta(\lambda/\rho-p\beta)$ for $\lambda>0$ and
$\beta\ge0$. Differentiating once more gives
\begin{equation}\label{eq:strict-mode-calculation}
  \begin{aligned}
    \mathcal R[\rho^{-\lambda}y^\beta]
     & =\rho^{-\lambda}y^\beta
    \left[(p-1)\left(\frac{\lambda}{\rho}-p\beta\right)^2
      +\frac{(p-1)\lambda}{\rho^2}
      +(n-p)\left(\frac{\lambda}{\rho}-p\beta\right)\right].
  \end{aligned}
\end{equation}
Write $\lambda_i=a+d_i$. Substituting
$(\lambda,\beta)=(\lambda_i,\beta_i)$ into
\eqref{eq:strict-mode-calculation} and using the minus sign in
\eqref{eq:strict-h-expanded} gives, for $i=1,2$,
\begin{equation*}
  \begin{aligned}
    \mathcal R[-\rho^{-\lambda_i}y^{\beta_i}]
    ={} & \rho^{-\lambda_i}y^{\beta_i}
    \biggl[p\beta_i\bigl(n-p-p(p-1)\beta_i\bigr)                 \\
        & \quad+\frac{\lambda_i\bigl(2p(p-1)\beta_i-(n-p)\bigr)}{\rho}
      -\frac{(p-1)(\lambda_i^2+\lambda_i)}{\rho^2}\biggr].
  \end{aligned}
\end{equation*}
By \eqref{eq:strict-power-signs},
$p\beta_i(n-p-p(p-1)\beta_i)$ is positive for $i=1$ and strictly
negative for $i=2$. The terms proportional to $\rho^{-1}$ and
$\rho^{-2}$ tend to zero, so the coefficient of
$\rho^{-\lambda_1}y^{\beta_1}$ is bounded above and the coefficient
of $\rho^{-\lambda_2}y^{\beta_2}$ is bounded above by a negative
constant. For $\rho^{-a}$, \eqref{eq:strict-mode-calculation} gives
\begin{equation*}
  \mathcal R[\rho^{-a}]
  =\rho^{-a}\left[
    \frac{(n-p)a}{\rho}+\frac{(p-1)(a^2+a)}{\rho^2}
    \right]
  \le C\rho^{-a-1}.
\end{equation*}
Combining the estimates for $\mathcal R[\rho^{-a}]$ and
$\mathcal R[-\rho^{-\lambda_i}y^{\beta_i}]$, $i=1,2$, gives
\begin{equation}\label{eq:strict-spatial-bracket}
  \mathcal R[h]\le \rho^{-a}
  \left[C\left(\rho^{-1}+\rho^{-d_1}y^{\beta_1}\right)
  -c\rho^{-d_2}y^{\beta_2}\right].
\end{equation}
Multiplying \eqref{eq:strict-spatial-bracket} by
$r^{-p}|rh_r|^{p-2}>0$ in \eqref{eq:endpoint-radial} gives
\eqref{eq:strict-spatial-sign}.

\subsection{The double-critical derivative estimates}
\label{app:critical-derivatives}

For the function $h$ in
\eqref{eq:critical-h}, that is,
\begin{equation*}
  h(r,t)=(\log\rho)^{-\nicefrac{1}{4}}
  \left[1-\varepsilon\rho^{-\nicefrac{1}{2}}y^{\nicefrac{1}{2}}
  -\varepsilon y^{\frac{n+1}{2}}
  \left(1+\frac{1+\log(1/y)}{\sqrt{\log\rho}}\right)^2\right],
\end{equation*}
we establish \eqref{eq:critical-time}, \eqref{eq:critical-radial-bound},
and \eqref{eq:critical-spatial}. In the current notation they read as %
\begin{align*}
  h_t&\ge\frac{c}{r^p(\log\rho)^{\nicefrac{1}{4}}}
  \left[\rho^{-\nicefrac{1}{2}}y^{-\nicefrac{1}{2}}
  +y^{\frac{n-1}{2}}
  \left(1+\frac{1+\log(1/y)}{\sqrt{\log\rho}}\right)^2\right]>0,\\
rh_r=  Dh&\ge\frac{c}{(\log\rho)^{\nicefrac{1}{4}}}
  \left[\frac1{\rho\log\rho}+\rho^{-\nicefrac{1}{2}}y^{\nicefrac{1}{2}}
  +y^{\frac{n+1}{2}}
  \left(1+\frac{1+\log(1/y)}{\sqrt{\log\rho}}\right)^2\right]>0,
\end{align*}
and
\begin{equation*}
  |\Deltap h|=r^{-p}|Dh|^{p-2}|\mathcal R[h]|\le
  \frac{Cr^{-p}|rh_r|^{p-2}}{(\log\rho)^{\nicefrac{1}{4}}}
  \left[\frac1{\rho\log\rho}+\rho^{-\nicefrac{1}{2}}y^{\nicefrac{1}{2}}
  +\frac{y^{\frac{n+1}{2}}\bigl(\sqrt{\log\rho}+1+\log(1/y)\bigr)}
  {\log\rho}\right].
\end{equation*}

To prove the first estimate, \eqref{eq:critical-time}, we apply the chain
rule to \eqref{eq:critical-h}. For the last term in
\eqref{eq:critical-h}, we need a lower bound for the $y$-derivative of
its positive $y$-dependent factor. We first differentiate its logarithm,
which turns the product into a sum, and then recover the derivative of
the factor itself:
\begin{equation*}
  y\partial_y\log\left[y^{\frac{n+1}{2}}
    \left(1+\frac{1+\log(1/y)}{\sqrt{\log\rho}}\right)^2\right]
  =\frac{n+1}{2}-\frac{2}{\sqrt{\log\rho}+1+\log(1/y)}\ge c>0.
\end{equation*}
By the chain rule, $\partial_y f(y)=f(y)\partial_y\log f(y)$ for
$f(y)>0$. Thus, multiplying this lower bound by the positive
expression inside the logarithm and dividing by $y$ gives
\begin{equation*}
  \partial_y\left[y^{\frac{n+1}{2}}
    \left(1+\frac{1+\log(1/y)}{\sqrt{\log\rho}}\right)^2\right]
  \ge c\,y^{\frac{n-1}{2}}
    \left(1+\frac{1+\log(1/y)}{\sqrt{\log\rho}}\right)^2.
\end{equation*}
For the second term in \eqref{eq:critical-h}, we use
$y\partial_y(\rho^{-\nicefrac{1}{2}}y^{\nicefrac{1}{2}})
=\frac12\rho^{-\nicefrac{1}{2}}y^{\nicefrac{1}{2}}$.
Since $y_t=-r^{-p}$ by \eqref{eq:singular-coordinates}, both
subtracted terms in \eqref{eq:critical-h} contribute positively to
$h_t$. The second and last terms of $h$ give the first and second
terms, respectively, in the brackets in \eqref{eq:critical-time},
proving that estimate.

To prove the second estimate, \eqref{eq:critical-radial-bound}, we
differentiate \eqref{eq:critical-h}, reusing the same chain-rule
identity in the form $Df=fD\log f$ for positive $f$. For the
prefactor and the factors inside the second and last terms,
respectively, we obtain
\begin{equation}\label{eq:critical-log-derivatives}
\begin{aligned}
  D\log(\log\rho)^{-\nicefrac{1}{4}}
   &=\frac1{4\rho\log\rho},\\
  D\log(\rho^{-\nicefrac{1}{2}}y^{\nicefrac{1}{2}})
   &=-\frac p2+\frac1{2\rho},\\
  D\log\left[y^{\frac{n+1}{2}}
    \left(1+\frac{1+\log(1/y)}{\sqrt{\log\rho}}\right)^2\right]
   &=-\frac{p(n+1)}2+\frac{2p}{\sqrt{\log\rho}+1+\log(1/y)}\\
   &\quad+\frac{1+\log(1/y)}{\rho\log\rho
     (\sqrt{\log\rho}+1+\log(1/y))}.
\end{aligned}
\end{equation}
For sufficiently large $\rho$, the second and third expressions in
\eqref{eq:critical-log-derivatives} are bounded above by negative
constants, whereas the first is positive. The bracket in
\eqref{eq:critical-h} is at least $1/2$. Thus the derivative of
$(\log\rho)^{-\nicefrac{1}{4}}$ supplies the first term,
$(\rho\log\rho)^{-1}$, in the brackets in
\eqref{eq:critical-radial-bound}. The second and third derivatives in
\eqref{eq:critical-log-derivatives}, together with the minus signs in
\eqref{eq:critical-h}, supply the second and third terms there.
This proves \eqref{eq:critical-radial-bound}.

It remains to prove \eqref{eq:critical-spatial}, and we do this by estimating
$\mathcal R[h]$ term by term.
We begin with the last term of $h$ in \eqref{eq:critical-h}, expanding
its square:
\begin{equation}\label{eq:critical-expanded-square}
  \begin{aligned}
    (\log\rho)^{-\nicefrac{1}{4}}y^{\frac{n+1}{2}}
    \left(1+\frac{1+\log(1/y)}{\sqrt{\log\rho}}\right)^2
    &=y^{\frac{n+1}{2}}\bigl[(\log\rho)^{-\nicefrac{1}{4}}\\
    &\quad+2(\log\rho)^{-\nicefrac{3}{4}}(1+\log(1/y))\\
    &\quad+(\log\rho)^{-\nicefrac{5}{4}}(1+\log(1/y))^2\bigr].
  \end{aligned}
\end{equation}
We apply $\mathcal R$ separately to the three summands on the
right-hand side of \eqref{eq:critical-expanded-square}.
For these summands, take $(\kappa,k)=(1/4,0),(3/4,1),(5/4,2)$,
respectively, and apply the product rule with
$a=(\log\rho)^{-\kappa}$ and
$b=y^{\frac{n+1}{2}}(1+\log(1/y))^k$:
\begin{equation}\label{eq:critical-product-rule}
\mathcal R[ab]
  =a\mathcal R[b]+2(p-1)(Da)(Db)+b\mathcal R[a].
\end{equation}
We first estimate $a\mathcal R[b]$, the first term in
\eqref{eq:critical-product-rule}, in which $a$ is not differentiated.
At $p=p_*$, the identity $n-p=p(p-1)(n+1)/2$ gives for $\mathcal R[b]$
\begin{align}
\label{eq:critical-poly}
  \mathcal R&[y^{\frac{n+1}{2}}(1+\log(1/y))^k]\nonumber
  \\&=p^2(p-1)y^{\frac{n+1}{2}}
  \left[k(k-1)(1+\log(1/y))^{k-2}
    -\frac{n+1}{2}k(1+\log(1/y))^{k-1}\right].
\end{align}
For $k=0$, \eqref{eq:critical-poly} gives
$\mathcal R[y^{\frac{n+1}{2}}]=0$; for $k=1,2$, each term on its
right-hand side has a lower power of $1+\log(1/y)$ than the expression
being differentiated. Multiplying by the corresponding factors $a$ in
\eqref{eq:critical-expanded-square}, we bound the sum of the absolute
values of the $a\mathcal R[b]$ contributions by
\begin{equation*}
  \begin{aligned}
    &C(\log\rho)^{-\nicefrac{3}{4}}y^{\frac{n+1}{2}}
    +C(\log\rho)^{-\nicefrac{5}{4}}y^{\frac{n+1}{2}}(1+\log(1/y))\\
    &\quad=C(\log\rho)^{-\nicefrac{1}{4}}y^{\frac{n+1}{2}}
    \left(\frac1{\sqrt{\log\rho}}
      +\frac{1+\log(1/y)}{\log\rho}\right)\\
    &\quad=C(\log\rho)^{-\nicefrac{1}{4}}y^{\frac{n+1}{2}}
      \frac{\sqrt{\log\rho}+1+\log(1/y)}{\log\rho}.
  \end{aligned}
\end{equation*}
Here $C$ includes the fixed $\varepsilon$ from \eqref{eq:critical-h}.

We next estimate the second and third terms in
\eqref{eq:critical-product-rule}, namely $2(p-1)(Da)(Db)$ and
$b\mathcal R[a]$. Both contain derivatives of
$a=(\log\rho)^{-\kappa}$, for which we use
\begin{align*}
  D(\log\rho)^{-\kappa}
   &=\frac{\kappa(\log\rho)^{-\kappa}}{\rho\log\rho},\\
  D^2(\log\rho)^{-\kappa}
   &=\frac{\kappa(\log\rho)^{-\kappa}}{\rho^2\log\rho}
     \left(1+\frac{\kappa+1}{\log\rho}\right).
\end{align*}
For $Db$ in the second term of \eqref{eq:critical-product-rule},
the inequality $1+\log(1/y)\ge1$ also gives
$|D(y^{\frac{n+1}{2}}(1+\log(1/y))^k)|
\le C y^{\frac{n+1}{2}}(1+\log(1/y))^k$.
The second term is therefore bounded using $Da$ and $Db$, and the
third using $\mathcal R[a]=(p-1)D^2a+(n-p)Da$. Summing these bounds
over the three summands in \eqref{eq:critical-expanded-square} gives
an absolute contribution at most
\begin{equation*}
  \frac{C(\log\rho)^{-\nicefrac{1}{4}}}{\rho\log\rho}
  y^{\frac{n+1}{2}}\left(1+\frac{1+\log(1/y)}{\sqrt{\log\rho}}\right)^2
  \le\frac{C(\log\rho)^{-\nicefrac{1}{4}}}{\rho\log\rho}.
\end{equation*}
The bound
$y^{\frac{n+1}{2}}(1+(1+\log(1/y))/\sqrt{\log\rho})^2\le C$
used in the second inequality follows
from the estimate after \eqref{eq:critical-h}.

This completes the estimate for the contribution of the last term of
$h$ to $\mathcal R[h]$. It remains to estimate
the contributions of the first and second terms of $h$  to
$\mathcal R[h]=(p-1)D^2h+(n-p)Dh$ i.e.\ estimate $Dh,\ D^2h$ in the case of these terms.
The first, $(\log\rho)^{-\nicefrac{1}{4}}$, contributes at most
$C(\log\rho)^{-\nicefrac{1}{4}}/(\rho\log\rho)$ to $\mathcal R[h]$.
For the second term of $h$ in \eqref{eq:critical-h}, the product rule gives
\begin{equation*}
  D\left[(\log\rho)^{-\nicefrac{1}{4}}
    \rho^{-\nicefrac{1}{2}}y^{\nicefrac{1}{2}}\right]
  =(\log\rho)^{-\nicefrac{1}{4}}
    \rho^{-\nicefrac{1}{2}}y^{\nicefrac{1}{2}}
  \left(-\frac p2+\frac1{2\rho}+\frac1{4\rho\log\rho}\right).
\end{equation*}
Differentiating once more gives
\begin{equation*}
\begin{aligned}
  &D^2\left[(\log\rho)^{-\nicefrac{1}{4}}
    \rho^{-\nicefrac{1}{2}}y^{\nicefrac{1}{2}}\right]\\
  &\quad=(\log\rho)^{-\nicefrac{1}{4}}
    \rho^{-\nicefrac{1}{2}}y^{\nicefrac{1}{2}}
  \left[\left(-\frac p2+\frac1{2\rho}+\frac1{4\rho\log\rho}\right)^2
    +\frac1{2\rho^2}
    +\frac{\log\rho+1}{4\rho^2(\log\rho)^2}\right].
\end{aligned}
\end{equation*}
The factor in parentheses in the first identity and the factor in brackets
in the second are bounded for sufficiently large $\rho$.
By definition of  $\mathcal R$, these identities imply
\begin{equation*}
  \left|\mathcal R\left[-\varepsilon(\log\rho)^{-\nicefrac{1}{4}}
    \rho^{-\nicefrac{1}{2}}y^{\nicefrac{1}{2}}\right]\right|
  \le C(\log\rho)^{-\nicefrac{1}{4}}
    \rho^{-\nicefrac{1}{2}}y^{\nicefrac{1}{2}}.
\end{equation*}
Combining the bounds for the first, second, and last terms of $h$ in
\eqref{eq:critical-h} gives
\begin{equation}\label{eq:critical-R-bound}
  |\mathcal R[h]|
  \le C(\log\rho)^{-\nicefrac{1}{4}}
  \left[\frac1{\rho\log\rho}
    +\rho^{-\nicefrac{1}{2}}y^{\nicefrac{1}{2}}
    +y^{\frac{n+1}{2}}\frac{\sqrt{\log\rho}+1+\log(1/y)}{\log\rho}\right].
\end{equation}
Multiplying \eqref{eq:critical-R-bound} by
$r^{-p}|rh_r|^{p-2}$ proves \eqref{eq:critical-spatial}.

\subsection{The double-critical comparison estimates}
\label{app:critical-comparisons}

Dividing the upper bound \eqref{eq:critical-spatial} by $h_t>0$ gives
\begin{equation*}
  \frac{|\Deltap h|}{h_t}
  \le \frac{Cr^{-p}|rh_r|^{p-2}}{h_t(\log\rho)^{\nicefrac{1}{4}}}
  \left[\frac1{\rho\log\rho}
    +\rho^{-\nicefrac{1}{2}}y^{\nicefrac{1}{2}}
    +y^{\frac{n+1}{2}}
      \frac{\sqrt{\log\rho}+1+\log(1/y)}{\log\rho}\right].
\end{equation*}
This is \eqref{eq:main-log}. We now estimate the contributions on the right-hand side.
To bound this right-hand side, we use the lower bounds for $h_t$ and
$rh_r$ in \eqref{eq:critical-time} and \eqref{eq:critical-radial-bound}.
These give upper bounds for $1/h_t$ and $|rh_r|^{p-2}$, respectively,
since $p-2<0$. For each contribution, we retain suitable terms from
these lower bounds.
For the repeated logarithmic factor, write
\begin{equation*}
  Q:=1+\frac{1+\log(1/y)}{\sqrt{\log\rho}}.
\end{equation*}
It holds that  $Q\ge1$ and
$\sqrt{\log\rho}+1+\log(1/y)=\sqrt{\log\rho}\,Q$.
The bounds below hold uniformly for $0<y<1$ and sufficiently large
$\rho$.

\smallskip
\noindent\emph{Third contribution in \eqref{eq:main-log}.}
The third term in the brackets equals
$y^{\frac{n+1}{2}}Q/\sqrt{\log\rho}$. To estimate its full
contribution, we retain the last term in each of
\eqref{eq:critical-time} and \eqref{eq:critical-radial-bound}, giving
\begin{equation*}
  h_t\ge c r^{-p}(\log\rho)^{-\nicefrac{1}{4}}
    y^{\frac{n-1}{2}}Q^2,
  \qquad
  rh_r\ge c(\log\rho)^{-\nicefrac{1}{4}}
    y^{\frac{n+1}{2}}Q^2.
\end{equation*}
Since $p-2<0$, the
lower bound for $rh_r$ controls $|rh_r|^{p-2}$ from above.
Substituting these two lower bounds in the full contribution gives
\begin{equation}\label{eq:critical-contribution-appendix}
\begin{aligned}
  &\frac{Cr^{-p}|rh_r|^{p-2}}{h_t(\log\rho)^{\nicefrac{1}{4}}}
    \frac{y^{\frac{n+1}{2}}Q}{\sqrt{\log\rho}}\\
  &\quad\le C(\log\rho)^{-\nicefrac{p}{4}}
    y^{1-\frac{n+1}{2}(2-p)}Q^{2p-5}\\
  &\quad=C(\log\rho)^{-\nicefrac{p}{4}}Q^{2p-5}
    \le C(\log\rho)^{-\nicefrac{p}{4}}.
\end{aligned}
\end{equation}
Here $\frac{n+1}{2}(2-p)=1$ cancels the power of $y$, and
$2p-5<0$ together with $Q\ge1$ gives the last inequality.

\smallskip
\noindent\emph{First contribution in \eqref{eq:main-log}.}
To estimate the full contribution of $(\rho\log\rho)^{-1}$, we use
the $(\rho\log\rho)^{-1}$ term in
\eqref{eq:critical-radial-bound} and the
$\rho^{-\nicefrac{1}{2}}y^{-\nicefrac{1}{2}}$ term in
\eqref{eq:critical-time}. These give
\begin{equation*}
  rh_r\ge\frac{c}{(\log\rho)^{\nicefrac{1}{4}}\rho\log\rho},
  \qquad
  h_t\ge
  \frac{c r^{-p}\rho^{-\nicefrac{1}{2}}y^{-\nicefrac{1}{2}}}
    {(\log\rho)^{\nicefrac{1}{4}}}.
\end{equation*}
Since $p-2<0$, the $(\rho\log\rho)^{-1}$ contribution in
\eqref{eq:main-log} satisfies
\begin{equation}\label{eq:critical-radial-appendix}
  \frac{Cr^{-p}|rh_r|^{p-2}}
    {h_t(\log\rho)^{\nicefrac{1}{4}}\rho\log\rho}
  \le C(\log\rho)^{\frac{2-p}{4}}
    (\rho\log\rho)^{-(p-1)}
    \rho^{\nicefrac{1}{2}}y^{\nicefrac{1}{2}}.
\end{equation}
In Region I of \cref{fig:double-critical-geometry},
$y\le\rho^{-1}$ makes
$\rho^{\nicefrac{1}{2}}y^{\nicefrac{1}{2}}\le1$
in \eqref{eq:critical-radial-appendix}.

\smallskip
\noindent\emph{Second contribution in \eqref{eq:main-log}.}
For the full contribution of $\rho^{-\nicefrac{1}{2}}y^{\nicefrac{1}{2}}$,
we first obtain a common bound and then estimate its denominator in two
ways. The last two terms in the bracket in
\eqref{eq:critical-radial-bound} have sum
\begin{equation*}
  A:=\rho^{-\nicefrac{1}{2}}y^{\nicefrac{1}{2}}
    +y^{\frac{n+1}{2}}Q^2.
\end{equation*}
The two terms in the bracket in \eqref{eq:critical-time} have sum
$y^{-1}A$. Thus these estimates give
\begin{equation*}
  h_t\ge c r^{-p}(\log\rho)^{-\nicefrac{1}{4}}y^{-1}A,
  \qquad rh_r\ge c(\log\rho)^{-\nicefrac{1}{4}}A.
\end{equation*}
Since $p-2<0$, substituting these lower bounds in the full contribution gives
\begin{equation}\label{eq:lower-contribution}
  \frac{Cr^{-p}|rh_r|^{p-2}
    \rho^{-\nicefrac{1}{2}}y^{\nicefrac{1}{2}}}
    {h_t(\log\rho)^{\nicefrac{1}{4}}}
  \le
  \frac{C(\log\rho)^{\frac{2-p}{4}}
    y\,\rho^{-\nicefrac{1}{2}}y^{\nicefrac{1}{2}}}
    {A^{3-p}}.
\end{equation}

For use when $y\le\rho^{-\frac{1}{2n}}$ (Regions I and II of
\cref{fig:double-critical-geometry}), we retain a negative power of $Q$
in the bound. Split $A^{3-p}=A\cdot A^{2-p}$ and use
$A\ge\rho^{-\nicefrac{1}{2}}y^{\nicefrac{1}{2}}$
for the first factor and
$A\ge y^{\frac{n+1}{2}}Q^2$ for $A^{2-p}$.
Since $\frac{n+1}{2}(2-p)=1$, we obtain
\begin{align}
  \frac{C(\log\rho)^{\frac{2-p}{4}}
    y\,\rho^{-\nicefrac{1}{2}}y^{\nicefrac{1}{2}}}
    {A^{3-p}}
  &\le
  \frac{C(\log\rho)^{\frac{2-p}{4}}
    y\,\rho^{-\nicefrac{1}{2}}y^{\nicefrac{1}{2}}}
    {(\rho^{-\nicefrac{1}{2}}y^{\nicefrac{1}{2}})
     (y^{\frac{n+1}{2}}Q^2)^{2-p}}
    \notag\\
  &=C(\log\rho)^{\frac{2-p}{4}}Q^{-2(2-p)}.
  \label{eq:critical-lower-small-denominator}
\end{align}

\begin{samepage}
For use when $y>\rho^{-\frac{1}{2n}}$ (Region III of
\cref{fig:double-critical-geometry}), we instead seek a bound in terms
of powers of $\rho$ and $y$. We bound the whole denominator using
$A\ge y^{\frac{n+1}{2}}Q^2\ge y^{\frac{n+1}{2}}$.
The identity
$\frac32-\frac{n+1}{2}(3-p)=-\frac n2$ gives
\begin{equation}\label{eq:critical-lower-large-denominator}
  \frac{C(\log\rho)^{\frac{2-p}{4}}
    y\,\rho^{-\nicefrac{1}{2}}y^{\nicefrac{1}{2}}}
    {A^{3-p}}
  \le \frac{C(\log\rho)^{\frac{2-p}{4}}
    y\,\rho^{-\nicefrac{1}{2}}y^{\nicefrac{1}{2}}}
    {\bigl(y^{\frac{n+1}{2}}\bigr)^{3-p}}
  = C(\log\rho)^{\frac{2-p}{4}}
    \rho^{-\nicefrac{1}{2}}y^{-\nicefrac{n}{2}}.
\end{equation}
\end{samepage}

\renewcommand{\bibliofont}{\footnotesize\setlength{\baselineskip}{10pt}}
\bibliographystyle{alpha}
\bibliography{refs}

@article{BB26,
  author        = {Bj{\"o}rn, A. and Bj{\"o}rn, J.},
  title         = {Barriers, {Barenblatt} solutions and regularity of soda can
                   domains for the heat equation and nonlinear
                   {$p$}-parabolic equations},
  year          = {2026},
  journal       = {J. Evol. Equ.},
  note          = {To appear; arXiv:2604.27516},
  eprint        = {2604.27516},
  archiveprefix = {arXiv},
  url           = {https://arxiv.org/abs/2604.27516},
}

@article{BBG17,
  author  = {Bj{\"o}rn, A. and Bj{\"o}rn, J. and Gianazza, U.},
  title   = {The {Petrovski\u\i} criterion and barriers for degenerate and
             singular {$p$}-parabolic equations},
  journal = {Math. Ann.},
  volume  = {368},
  year    = {2017},
  pages   = {885--904},
  doi     = {10.1007/s00208-016-1415-0},
}

@article{BBGP15,
  author  = {Bj{\"o}rn, A. and Bj{\"o}rn, J. and Gianazza, U. and
             Parviainen, M.},
  title   = {Boundary regularity for degenerate and singular parabolic
             equations},
  journal = {Calc. Var. Partial Differential Equations},
  volume  = {52},
  year    = {2015},
  pages   = {797--827},
  doi     = {10.1007/s00526-014-0734-9},
}

@article{KL96,
  author  = {Kilpel{\"a}inen, T. and Lindqvist, P.},
  title   = {On the {Dirichlet} boundary value problem for a degenerate
             parabolic equation},
  journal = {SIAM J. Math. Anal.},
  volume  = {27},
  year    = {1996},
  pages   = {661--683},
  doi     = {10.1137/0527036},
}

@article{KKP10,
  author  = {Korte, R. and Kuusi, T. and Parviainen, M.},
  title   = {A connection between a general class of superparabolic functions
             and supersolutions},
  journal = {J. Evol. Equ.},
  volume  = {10},
  year    = {2010},
  pages   = {1--20},
  doi     = {10.1007/s00028-009-0037-3},
}

@article{Perron23,
  author  = {Perron, O.},
  title   = {{Eine neue Behandlung der ersten Randwertaufgabe f{\"u}r
             $\Delta u=0$}},
  journal = {Math. Z.},
  volume  = {18},
  year    = {1923},
  pages   = {42--54},
  doi     = {10.1007/BF01192395},
}

@article{Sternberg29,
  author  = {Sternberg, W.},
  title   = {{{\"U}ber die Gleichung der W{\"a}rmeleitung}},
  journal = {Math. Ann.},
  volume  = {101},
  year    = {1929},
  pages   = {394--398},
  doi     = {10.1007/BF01454850},
}

@article{Wiener24,
  author  = {Wiener, N.},
  title   = {The {Dirichlet} problem},
  journal = {J. Math. Phys.},
  volume  = {3},
  year    = {1924},
  pages   = {127--146},
  doi     = {10.1002/sapm192433127},
}

@article{Lan73,
  author  = {Lanconelli, E.},
  title   = {Sul problema di {Dirichlet} per l'equazione del calore},
  journal = {Ann. Mat. Pura Appl.},
  volume  = {97},
  year    = {1973},
  pages   = {83--114},
  doi     = {10.1007/BF02414910},
}

@article{EG82,
  author  = {Evans, L. C. and Gariepy, R. F.},
  title   = {{Wiener}'s criterion for the heat equation},
  journal = {Arch. Rational Mech. Anal.},
  volume  = {78},
  year    = {1982},
  pages   = {293--314},
  doi     = {10.1007/BF00249583},
}

@article{GLM86,
  author  = {Granlund, S. and Lindqvist, P. and Martio, O.},
  title   = {Note on the {PWB}-method in the non-linear case},
  journal = {Pacific J. Math.},
  volume  = {125},
  year    = {1986},
  pages   = {381--395},
  doi     = {10.2140/pjm.1986.125.381},
}

@article{Petrovskii35,
  author  = {Petrowsky, I.},
  title   = {{Zur ersten Randwertaufgabe der W{\"a}rmeleitungsgleichung}},
  journal = {Compos. Math.},
  volume  = {1},
  year    = {1935},
  pages   = {383--419},
  url     = {https://archive.numdam.org/item/CM_1935__1__383_0/},
}

@article{L95,
  author  = {Lindqvist, P.},
  title   = {A criterion of {Petrowsky}'s kind for a degenerate
             quasilinear parabolic equation},
  journal = {Rev. Mat. Iberoam.},
  volume  = {11},
  year    = {1995},
  pages   = {569--578},
  doi     = {10.4171/RMI/185},
}

@article{EK71,
  author  = {Effros, E. G. and Kazdan, J. L.},
  title   = {On the {Dirichlet} problem for the heat equation},
  journal = {Indiana Univ. Math. J.},
  volume  = {20},
  year    = {1971},
  pages   = {683--693},
  doi     = {10.1512/iumj.1971.20.20054},
}

@article{Abdulla05,
  author  = {Abdulla, U. G.},
  title   = {Multidimensional {Kolmogorov--Petrovsky} test for the
             boundary regularity and irregularity of solutions to the
             heat equation},
  journal = {Bound. Value Probl.},
  volume  = {2005},
  year    = {2005},
  pages   = {181--199},
  doi     = {10.1155/BVP.2005.181},
}

@book{DGV12,
  author    = {DiBenedetto, E. and Gianazza, U. and Vespri, V.},
  title     = {{Harnack}'s Inequality for Degenerate and Singular
               Parabolic Equations},
  series    = {Springer Monographs in Mathematics},
  publisher = {Springer},
  address   = {New York},
  year      = {2012},
  doi       = {10.1007/978-1-4614-1584-8},
}

@article{GLL18,
  author  = {Gianazza, U. and Liao, N. and Lukkari, T.},
  title   = {A boundary estimate for singular parabolic diffusion equations},
  journal = {NoDEA Nonlinear Differential Equations Appl.},
  volume  = {25},
  number  = {4},
  year    = {2018},
  note    = {Paper No. 33},
  doi     = {10.1007/s00030-018-0523-9},
}

@article{GL20,
  author  = {Gianazza, U. and Liao, N.},
  title   = {A boundary estimate for degenerate parabolic diffusion equations},
  journal = {Potential Anal.},
  volume  = {53},
  year    = {2020},
  pages   = {977--995},
  doi     = {10.1007/s11118-019-09794-8},
}

@article{GL22,
  author  = {Gianazza, U. and Liao, N.},
  title   = {A boundary estimate for singular sub-critical parabolic equations},
  journal = {Int. Math. Res. Not. IMRN},
  volume  = {2022},
  number  = {10},
  year    = {2022},
  pages   = {7332--7353},
  doi     = {10.1093/imrn/rnaa351},
}

@article{MS25,
  author  = {Moring, K. and Scheven, C.},
  title   = {On notions of {$p$}-parabolic capacity and applications},
  journal = {Potential Anal.},
  volume  = {63},
  year    = {2025},
  pages   = {895--946},
  doi     = {10.1007/s11118-025-10193-5},
}

@article{BG15,
  author  = {Banerjee, A. and Garofalo, N.},
  title   = {On the {Dirichlet} boundary value problem for the normalized
             {$p$}-{Laplacian} evolution},
  journal = {Commun. Pure Appl. Anal.},
  volume  = {14},
  number  = {1},
  year    = {2015},
  pages   = {1--21},
  doi     = {10.3934/cpaa.2015.14.1},
}

@article{BBP19,
  author  = {Bj{\"o}rn, A. and Bj{\"o}rn, J. and Parviainen, M.},
  title   = {The tusk condition and {Petrovski\u\i} criterion for the
             normalized {$p$}-parabolic equation},
  journal = {J. Lond. Math. Soc. (2)},
  volume  = {100},
  number  = {2},
  year    = {2019},
  pages   = {623--643},
  doi     = {10.1112/jlms.12224},
}

@misc{K24,
  author        = {Kurkinen, T.},
  title         = {Boundary regularity for a general nonlinear parabolic
                   equation in non-divergence form},
  year          = {2024},
  note          = {Preprint, arXiv:2404.12848},
  eprint        = {2404.12848},
  archiveprefix = {arXiv},
  url           = {https://arxiv.org/abs/2404.12848},
}

@article{KKKP13,
  author  = {Kinnunen, J. and Korte, R. and Kuusi, T. and Parviainen, M.},
  title   = {Nonlinear parabolic capacity and polar sets of superparabolic
             functions},
  journal = {Math. Ann.},
  volume  = {355},
  number  = {4},
  year    = {2013},
  pages   = {1349--1381},
  doi     = {10.1007/s00208-012-0825-x},
}

@article{AKP15,
  author  = {Avelin, B. and Kuusi, T. and Parviainen, M.},
  title   = {Variational parabolic capacity},
  journal = {Discrete Contin. Dyn. Syst.},
  volume  = {35},
  number  = {12},
  year    = {2015},
  pages   = {5665--5688},
  doi     = {10.3934/dcds.2015.35.5665},
}

@article{A16,
  author  = {Avelin, B.},
  title   = {On time dependent domains for the degenerate {$p$}-parabolic
             equation: {Carleson} estimate and {H{\"o}lder} continuity},
  journal = {Math. Ann.},
  volume  = {364},
  number  = {1--2},
  year    = {2016},
  pages   = {667--686},
  doi     = {10.1007/s00208-015-1226-8},
}

@article{AKN19,
  author  = {Avelin, B. and Kuusi, T. and Nystr{\"o}m, K.},
  title   = {Boundary behavior of solutions to the parabolic {$p$}-{Laplace}
             equation},
  journal = {Anal. PDE},
  volume  = {12},
  number  = {1},
  year    = {2019},
  pages   = {1--42},
  doi     = {10.2140/apde.2019.12.1},
}

@article{DPP03,
  author  = {Droniou, J. and Porretta, A. and Prignet, A.},
  title   = {Parabolic capacity and soft measures for nonlinear equations},
  journal = {Potential Anal.},
  volume  = {19},
  number  = {2},
  year    = {2003},
  pages   = {99--161},
  doi     = {10.1023/A:1023248531928},
}

@article{AGS16,
  author  = {Avelin, B. and Gianazza, U. and Salsa, S.},
  title   = {Boundary estimates for certain degenerate and singular
             parabolic equations},
  journal = {J. Eur. Math. Soc. (JEMS)},
  volume  = {18},
  number  = {2},
  year    = {2016},
  pages   = {381--424},
  doi     = {10.4171/JEMS/593},
}

@article{KM94,
  author  = {Kilpel{\"a}inen, T. and Mal{\'y}, J.},
  title   = {The {Wiener} test and potential estimates for quasilinear
             elliptic equations},
  journal = {Acta Math.},
  volume  = {172},
  number  = {1},
  year    = {1994},
  pages   = {137--161},
  doi     = {10.1007/BF02392793},
}

\end{document}